\documentclass[11pt,reqno,a4paper]{amsart}

\usepackage[left=3.5cm,right=3.5cm,top=3cm,bottom=3cm]{geometry}
\usepackage{amsmath}
\usepackage{amsfonts,amssymb}
\usepackage{enumerate}

\usepackage{amsthm}

\usepackage{hyperref} 
\hypersetup{
    colorlinks=true,       
    linkcolor=red,          
    citecolor=black,        
    filecolor=magenta,      
    urlcolor=cyan           
}

\newtheorem{theorem}{Theorem}[section]
\newtheorem{corollary}[theorem]{Corollary}
\newtheorem{lemma}[theorem]{Lemma}
\newtheorem{prop}[theorem]{Proposition}
\theoremstyle{definition}
\newtheorem{definition}[theorem]{Definition}
\newtheorem{notation}[theorem]{Notation}
\newtheorem*{notation*}{Notation}
\newtheorem{remark}[theorem]{Remark}

\DeclareMathOperator{\cn}{div}

\DeclareMathOperator{\di}{d}
\DeclareMathOperator{\sign}{sgn}

\DeclareMathOperator{\pv}{pv}

\def\eps{\epsilon}

\def\be{\begin{equation}}
\def\blA{\bigl\lVert}
\def\brA{\bigr\rVert}
\def\e{\eqref}
\def\ee{\end{equation}}

\def\dalpha{\di\!\alpha}

\def\deta{\di\! \eta}
\def\dmu{\di\! \mu}
\def\dxi{\di\!\xi}
\def\dt{\di\! t}
\def\dx{\di\! x}
\def\dy{\di\! y}
\def\fract{\frac{\di}{\dt}}

\def\defn{\mathrel{:=}}
\def\la{\left\lvert}
\def\lA{\left\lVert}
\def\le{\leq}
\def\les{\lesssim}

\def\mez{\frac{1}{2}}
\def\ra{\right\rvert}
\def\rA{\right\rVert}
\def\tdm{\frac{3}{2}}

\def\xN{\mathbb{N}}
\def\xR{\mathbb{R}}

\numberwithin{equation}{section}

\numberwithin{equation}{section}
\def\notebasdepage[#1]#2{\begingroup\def\thefootnote{\fnsymbol{footnote}}\footnote[#1]{#2}\endgroup}

\begin{document}

\date{}
\title{Paralinearization of the Muskat equation and application to the Cauchy problem} 

\author{Thomas Alazard}
\address{}

\author{Omar Lazar}

\maketitle

\begin{abstract}
We paralinearize the Muskat equation to extract an explicit parabolic evolution equation 
having a compact form. This result is applied to give a simple proof of 
the local well-posedness 
of the Cauchy problem for rough initial data, in homogeneous Sobolev spaces $\dot{H}^1(\xR)\cap \dot{H}^s(\xR)$ with $s>3/2$. 
This paper is essentially self-contained and does not rely on general 
results from paradifferential calculus.
\end{abstract}

\bibliographystyle{plain}

\section{Introduction}

The Muskat equation is a fundamental equation for 
incompressible fluids in porous media. It describes the evolution of a time-dependent free surface $\Sigma(t)$ separating two 
fluid domains $\Omega_1(t)$ and $\Omega_2(t)$. 
A common assumption in this theory is that the motion is in two dimensions so that the interface is a curve. 
In this introduction, for the sake of simplicity, we assume that the 
interface is a graph (the analysis is done later on for a general interface). 
On the supposition that the fluids extend indefinitely in horizontal directions, it results that
\begin{align*}
\Omega_1(t)&=\left\{ (x,y)\in \xR\times \xR\,;\, y>h(t,x)\right\},\\
\Omega_2(t)&=\left\{ (x,y)\in \xR\times \xR\,;\, y<h(t,x)\right\},\\
\Sigma(t)&=\partial \Omega_1(t)=\partial \Omega_2(t)=\{y=h(t,x)\}.
\end{align*}
Introduce the density $\rho_i$, the velocity $v_i$ and the pressure $P_i$ in the domain $\Omega_i$ ($i=1,2$). 
One assumes that the velocities $v_1$ and $v_2$ obey Darcy's law. 
Then, the equations by which the motion is to be determined are
\begin{alignat*}{3}
v_i&=\nabla (P_i+\rho_i g y) \qquad&&\text{in }&&\Omega_i,\\
\cn v_i&=0 && \text{in }&&\Omega_i,\\
P_1&=P_2 &&\text{on }&&\Sigma,\\
v_1\cdot n &=v_2\cdot n  &&\text{on }&&\Sigma,
\end{alignat*}
where $g$ is the gravity and $n$ is the outward unit normal to $\Omega_2$ on $\Sigma$,
$$
n=\frac{1}{\sqrt{1+(\partial_x h)^2}} \begin{pmatrix} -\partial_x h \\ 1\end{pmatrix}.
$$
The first two equations express the classical Darcy's law 
and the last two equations impose the continuity of the pressure and 
the normal velocities at the interface. This system is supplemented with an equation 
for the evolution of the free 
surface:
$$
\partial_t h=\sqrt{1+(\partial_x h)^2}\, v_2 \cdot n.
$$
The previous system has been introduced by Muskat 
in~\cite{Muskat} whose main application was in petroleum engineering 
(see~\cite{Muskatbook,Narasimhan} for many historical comments). 

\bigbreak

In \cite{CG-CMP}, C\'ordoba and Gancedo  discovered a formulation of the previous system 
based 
on contour integral, 
which applies whether the interface is a graph or not. 
The latter work opened the door to the solution of many important problems concerning 
the Cauchy problem or blow-up solutions 
(see \cite{CCG-Annals,CCFG-ARMA-2013,CCFG-ARMA-2016,CCFGLF-Annals-2012}, 
more references are given below as well as in the survey papers~\cite{GancedoSEMA,GraneroLazar}). 
This formulation is a compact equation 
where the unknown is the parametrization of the free surface, namely 
a function $f=f(t,x)$ depending on time $t\in \xR_+$ and $x\in \xR$, satisfying
\begin{equation}\label{nM1}
\partial_{t}f= \frac{\rho}{2\pi}\partial_{x} \int \arctan \left(\Delta_{\alpha} f \right)\dalpha,
\end{equation}
where $\rho={\rho_2-\rho_1}$ is the difference of the densities, the integral 
is understood in the principal value sense and $\Delta_{\alpha} f$ is the 
slope, namely
$$
\Delta_{\alpha} f (t,x)=\frac{f(x,t)-f(x-\alpha,t)}{\alpha}\cdot
$$
The beauty of equation \e{nM1} lies in its apparent simplicity, which should be compared with the complexity 
of the equations written in 
Eulerian formulation. This might suggest that~\e{nM1} is the simplest version of the Muskat equation 
one may hope for. However, since the equation is highly nonlocal 
(this means that the nonlinearity enters in the nonlocal terms), 
even with this formulation the study of the Cauchy problem for \e{nM1} 
is a very delicate problem. 
We refer the reader to the above mentioned papers for the description 
of the main difficulties 
one has to cope with.

\smallbreak

Our goal in this paper is to continue this line of research. We want to simplify further 
the study of the Muskat problem by transforming the equation~\e{nM1} into the 
simplest possible form. 
We shall prove that one can derive from the formulation~\e{nM1} an explicit parabolic evolution. 
In particular, we shall see that one can 
decouple the nonlinear and nonlocal aspects. There are many 
possible applications that one could work out of this explicit parabolic formulation. 
Here we shall study the Cauchy problem in homogeneous Sobolev spaces.

\smallbreak

The well-posedness of the Cauchy problem was first proved in \cite{CG-CMP} 
by C\'ordoba and Gancedo for initial data in $H^3(\xR)$ 
in the stable regime $\rho_2>\rho_1$ (they also proved that 
the problem is ill-posed in Sobolev spaces when $\rho_2<\rho_1$). 
Several extensions of their results have been obtained by different proofs. 
In \cite{Cheng-Belinchon-Shkoller-AdvMath}, 
Cheng, Granero-Belinch\'on, Shkoller 
proved the well-posedness of the Cauchy problem in $H^2(\xR)$ 
(introducing a Lagrangian point of view which can be used in a 
broad setting, see~\cite{GraneroShkoller}) and 
Constantin, Gancedo, Shvydkoy and Vicol~(\cite{CGSV-AIHP2017}) 
considered rough initial data which are in $W^{2,p}(\xR)$ for some $p>1$, as well 
they obtained a regularity criteria for the Muskat problem. 
We refer also to the recent work \cite{GraneroLazar} where 
a regularity criteria is obtained in terms of a control of some critical quantities. 
Many recent results are motivated by the fact that, 
loosely speaking, the Muskat equation 
has to do with the slope more than with the curvature of the fluid interface. Indeed, 
one scale invariant norm is the Lipschitz norm $\sup_{x\in\xR}\la \partial_x f(t,x)\ra$. 
We refer the reader to the work \cite{CCGRPS-AJM2016} 
of Constantin, C{\'o}rdoba, Gancedo, Rodr{\'\i}guez-Piazza 
and Strain for global well-posedness results assuming 
that the Lipschitz semi-norm is smaller than $1$ (see also \cite{PSt} where time decay of those solutions is proved). 
In \cite{DLL}, Deng, Lei and Lin proved the existence 
of global in time solutions with large slopes, 
assuming some monotonicity assumption on the data. 
In \cite{Cameron}, Cameron was able to prove a global existence result assuming 
that some critical quantity, namely the product of the maximal and minimal slopes, 
is smaller than 1. His result allows to consider arbitrary large slopes. 
By using a new formulation of the Muskat equation involving oscillatory integrals, 
C\'ordoba and the second author in \cite{Cordoba-Lazar-H3/2} proved that 
the Muskat equation is globally well-posed for  sufficiently smooth data 
provided the critical Sobolev norm  $\dot H^{\tdm}(\xR)$ is small enough. The latter is a 
global existence result of a unique strong solution having arbitrarily large slopes. 

These observations suggest to study the local in time well-posedness of the Cauchy problem 
without assuming that any $L^p$-norm 
of the curvature is finite. The well-posedness of the Cauchy problem 
in this case was obtained by Matioc~\cite{Matioc1,Matioc2}. 
Using tools from functional analysis, Matioc 
proved that the Cauchy problem is locally in time well-posed for initial 
data in Sobolev spaces $H^s(\xR)$ with $s>3/2$, without smallness assumption. 
We shall give a simpler proof which generalizes the latter result 
to homogeneous Sobolev spaces $\dot{H}^s(\xR)$. 
Eventually, let us mention that many recent results 
focus on different rough solutions, which are important for instance in 
the unstable regime $\rho_1>\rho_2$ (see {\em e.g}.\  
the existence mixing zones in \cite{mix1,mix2,Otto} 
or the dynamic between the two different regimes~\cite{shift1,shift2}). 
We refer also to \cite{uniq1,uniq2} where uniqueness 
issues have been studied using the convex integration scheme. 

In this paper we assume that the difference between 
the densities in the two fluids satisfies $\rho>0$, so, 
by rescaling in time, we can assume without loss of generality that $\rho=2$. 

A fundamental difference with the above mentioned results is that we shall determine 
the full structure of the nonlinearity instead of performing energy estimates. 
To explain this, we begin by identifying the nonlinear terms. 
Since $\rho=2$, one can rewrite equation~\e{nM1} as
\begin{equation*}
\partial_tf=\frac{1}{\pi}\int \frac{\partial_x \Delta_\alpha f}{1+(\Delta_\alpha f)^2}\dalpha
\end{equation*}
(in this introduction some computations are formal, but 
we shall rigorously justify them later). 
Consequently, the linearized Muskat equation reads
\be\label{LinM}
\partial_t u=\frac{1}{\pi}\pv \int \partial_x \Delta_\alpha u\dalpha.
\ee
Consider the singular integral operators
\be\label{defiHilbertLambda}
\mathcal{H}u=-\frac{1}{\pi}\pv\int \Delta_\alpha u \dalpha\quad\text{and}\quad
\Lambda=\mathcal{H}\partial_x.
\ee
Then $\mathcal{H}$ is the Hilbert transform (the Fourier multiplier with symbol $-i\sign(\xi)$) and $\Lambda$ 
is the square root of $-\partial_{xx}$. 
With the latter notation, the linearized Muskat equation~\e{LinM} 
reads
$$
\partial_t u+ \Lambda u=0.
$$
With this notation, the  Muskat equation~\e{nM1} can be written under the form 
\be\label{nM3}
\partial_t f+ \Lambda f=\mathcal{T}(f)f,
\ee
where $\mathcal{T}(f)$ is the operator defined by
$$
\mathcal{T}(f)g=-\frac{1}{\pi}\int \big(\partial_x \Delta_\alpha g\big)
\frac{(\Delta_\alpha f)^2}{1+(\Delta_\alpha f)^2}\dalpha.
$$
Our first main result will provide a thorough study of this nonlinear operator.
 \ Before going any further, let us fix some notations. 
\begin{definition}
$i)$ Given a real number $\sigma$, we denote by $\Lambda^\sigma$ 
the Fourier multiplier with symbol $\la\xi\ra^\sigma$ and by $\dot{H}^\sigma(\xR)$ 
the homogeneous Sobolev space of tempered distributions whose 
Fourier transform $\hat{u}$ belongs to 
$L^1_{loc}(\xR)$ and satisfies
$$
\lA u\rA_{\dot{H}^\sigma}^2=
\blA \Lambda^\sigma u\brA_{L^2}^2=\frac{1}{2\pi}\int_\xR \la \xi\ra^{2\sigma}\la \hat{u}(\xi)\ra^2\dxi<+\infty.
$$

$ii)$ We denote by $H^\sigma(\xR)$ the nonhomogeneous Sobolev space $L^2(\xR)\cap \dot{H}^\sigma(\xR)$. 
We set $H^\infty(\xR)\defn \cap_{\sigma\ge 0}H^\sigma(\xR)$ and 
introduce $X\defn \cap_{\sigma\ge 1}\dot{H}^\sigma(\xR)$, the set 
of tempered distributions whose 
Fourier transform $\hat{u}$ belongs to 
$L^1_{loc}(\xR)$ and whose derivative belongs to $H^\infty(\xR)$.

$iii)$ Given $0<s<1$, the homogeneous Besov space $\dot B^{s}_{2,1}(\mathbb R)$ consists of those 
tempered distributions $f$
whose Fourier transform is integrable near the origin and such that 
$$
\lA f \rA_{\dot B^{s}_{2,1}} =
\int \left(\int \frac{\la f(x)-f(x-\alpha)\ra^2}{\la \alpha\ra^{2s}}\dx\right)^\mez \frac{\dalpha}{\la\alpha\ra}<+\infty.
$$

$iv)$ We use the notation 
$\lA \cdot\rA_{E\cap F}=\lA \cdot\rA_E+\lA \cdot\rA_F$.
\end{definition}
\begin{theorem}\label{T1} 
\begin{enumerate}[i)]
\item\label{T1:low} (Low frequency estimate) There exists 
a constant $C$ such that, for all $f$ in $\dot{H}^{1}(\xR)$ and all $g$ in $\dot{H}^\tdm(\xR)$, 
$\mathcal{T}(f)g$ belongs to $L^2(\xR)$ and
$$
\lA \mathcal{T}(f)g\rA_{L^2}\le C \lA f\rA_{\dot{H}^{1}}
\lA g\rA_{\dot{H}^{\tdm}}.
$$
Moreover, $f\mapsto \mathcal{T}(f)f$ is locally Lipschitz from 
$\dot{H}^1(\xR)\cap \dot{H}^{\tdm}(\xR)$ to $L^2(\xR)$.
\item\label{T1:high} (High frequency estimate) 
For all $0<\nu<\eps<1/2$, there exists a positive constant $C>0$  
such that, for all functions $f,g$ in $X=\cap_{\sigma\ge 1}\dot{H}^\sigma(\xR)$,
\be\label{n3}
\mathcal{T}(f)g=\gamma(f)\Lambda g+V(f)\partial_x g+R(f,g)
\ee
where 
$$
\gamma(f)\defn \frac{f_x^2}{1+f_x^2},
$$
and $R(f,g)$ and $V(f)$ satisfy
\be\label{n400}
\lA R(f,g)\rA_{L^2}\le C\lA f\rA_{\dot{H}^{\tdm+\eps}}\lA g\rA_{\dot{B}^{1-\eps}_{2,1}},
\ee
and
\be\label{n401}
\lA V(f)\rA_{C^{0,\nu}}\defn\lA V(f)\rA_{L^\infty}+\sup_{y\in \xR}\bigg(\frac{\la V(f)(x+y)-V(f)(x)\ra}{\la y\ra^\nu}\bigg)
\le C\lA f\rA_{\dot{H}^1\cap \dot{H}^{\frac{3}{2}+\eps}}^2.
\ee
\item\label{T1:commutator} Let $0<\eps<1/2$. 
There exists a non-decreasing function $\mathcal{F}\colon \xR_+\rightarrow \xR_+$ 
such that, for all functions $f,g$ in $X$,
\be\label{n3.1}
\lA\Lambda^{1+\eps}\mathcal{T}(f)g-\mathcal{T}(f)\Lambda^{1+\eps}g\rA_{L^2}
\le \mathcal{F}\big(\lA f\rA_{\dot{H}^1\cap \dot{H}^{\frac{3}{2}+\eps}}\big)
\lA f\rA_{\dot{H}^1\cap \dot{H}^{\frac{3}{2}+\eps}}\lA g \rA_{\dot H^{\tdm+\eps}\cap \dot H^2}.
\ee
\end{enumerate}
\end{theorem}
The proof of the first statement follows directly from the 
definition of fractional Sobolev spaces in terms 
of finite differences, see Section~\ref{S:preliminaries}. 
The proof of the second statement is the most 
delicate part of the proof, which requires to uncover 
some symmetries in the nonlinearity, see Section~\ref{S:high}. 
The last statement is proved in Section~\ref{S:commutator} 
by using sharp variants of the usual 
nonlinear estimates in Sobolev spaces. Namely we used for the later proof a version 
of the classical Kato-Ponce estimate proved recently by Li and also 
a refinement of the composition rule in Sobolev spaces proved in Section~\ref{S:preliminaries}. 

We deduce from the previous result a paralinearization formula for the nonlinearity. 
We do not consider paradifferential operators as introduced by Bony (\cite{Bony,MePise}). 
Instead, following Shnirelman~\cite{Shnirelman}, 
we consider a simpler version of these operators which is convenient for the analysis 
of the Muskat equation for rough solutions. 

\begin{corollary}\label{Coro:para}
Consider $0<\eps<1/2$ and, given a bounded function $a=a(x)$, denote by 
$\tilde{T}_a\colon \dot{H}^{1+\eps}(\xR)\rightarrow L^2(\xR)\cap \dot{H}^{1+\eps}(\xR)$ the paraproduct operator 
defined by
$$
\tilde{T}_ag=(I+\Lambda^{1+\eps})^{-1}(a\Lambda^{1+\eps}g).
$$
Then, there exists a function $\mathcal{F}\colon\xR\rightarrow\xR$ such that, for all $f\in X$,
\be\label{coro:paraeq}
\mathcal{T}(f)f=\tilde{T}_{\gamma(f)}\Lambda f+\tilde{T}_{V(f)}\partial_x f+
R_\eps(f),
\ee
where
\be\label{n408}
\lA R_\eps(f)\rA_{H^{1+\eps}}\le \mathcal{F}\big(\lA f\rA_{\dot{H}^1\cap \dot{H}^{\tdm+\eps}}\big)
\lA f\rA_{\dot{H}^1\cap \dot{H}^{\frac{3}{2}+\eps}}
\lA f \rA_{\dot H^{1}\cap \dot H^{2+\frac{\eps}{2}}}.
\ee
\end{corollary}
\begin{proof}
Writing
$$
\mathcal{T}(f)f=(I+\Lambda^{1+\eps})^{-1}\mathcal{T}(f)\Lambda^{1+\eps}f+(I+\Lambda^{1+\eps})^{-1}\big[\Lambda^{1+\eps},\mathcal{T}(f)\big]f,
$$
and using the formula~\e{n3} we find that \e{coro:paraeq} holds with
$$
R_\eps(f)=(I+\Lambda^{1+\eps})^{-1}R\big(f,\Lambda^{1+\eps}f\big)+(I+\Lambda^{1+\eps})^{-1}\big[\Lambda^{1+\eps},\mathcal{T}(f)\big]f.
$$
Then, it follows from~\e{n400} and \e{n3.1} that 
$$
\lA R_\eps(f)\rA_{{H}^{1+\eps}}\le \mathcal{F}\big(\lA f\rA_{\dot{H}^1\cap \dot{H}^{\tdm+\eps}}\big)
\lA f\rA_{\dot{H}^1\cap \dot{H}^{\frac{3}{2}+\eps}}\Big(
\lA f \rA_{\dot H^{\tdm+\eps}\cap \dot H^{2}}+\lA \Lambda^{1+\eps}f\rA_{\dot{B}^{1-\eps}_{2,1}}\Big).
$$
So~\e{n408} follows from 
$\dot{H}^{1-\frac{3\eps}{2}}(\xR)\cap \dot{H}^{1-\frac{\eps}{2}}(\xR)\hookrightarrow \dot{B}^{1-\eps}_{2,1}(\xR)$ 
(see Lemma~\ref{L:B21}).
\end{proof}

We now consider the Cauchy problem for the Muskat equation. 
Substituting the above identity for $\mathcal{T}(f)$ in the equation~\e{nM3} and simplifying, we find
$$
\Big(\partial_t -V(f)\partial_x +\frac{1}{1+f_x^2}\Lambda \Big)\Lambda^{1+\eps}f=\Lambda^{1+\eps}R_\eps(f).
$$
Now, the key point is that the estimates~\e{n400} and \e{n408} 
mean that the remainder term $R_\eps(f)$ and the operator $V\partial_x$ 
contribute as operators of order stricly less than $1$ (namely $1-\eps/2$ and $1-\nu$) 
to an energy estimate, and so they are 
sub-principal terms for the analysis of the Cauchy problem. 
We also observe that the Muskat equation 
is parabolic as long as one controls the $L^\infty$-norm of $f_x$ only. 
This observation is related to our second goal, which is to solve the Cauchy problem 
in homogeneous Sobolev spaces 
instead of nonhomogeneous spaces. 
This is a natural result since the Muskat equation is invariant by 
the transformation $f\mapsto f+C$. This allows us to make an assumption 
only on the $L^\infty$-norm of the slope of the initial data, allowing initial data 
which are not bounded or not square integrable.

\begin{theorem} \label{T:Cauchy}
Consider $s\in (3/2,2)$ and an initial data 
$f_0$ in $\dot{H}^1(\xR)\cap \dot{H}^s(\xR)$. 
Then, there exists a positive time $T$ such that the Cauchy problem for \e{nM1} with initial 
data $f_0$ has a unique solution $f$ satisfying $f(t,x)=f_0(x)+u(t,x)$ with $u(0,x)=0$ and
$$
u \in C^0\big([0,T]; H^s(\xR)\big)\cap 
C^1\big([0,T];H^{s-1}(\xR)\big)
\cap L^{2}\big(0,T; H^{s+\mez}(\xR)\big),
$$
where $H^\sigma(\xR)$ denotes the nonhomogeneous Sobolev space 
$L^2(\xR)\cap \dot{H}^\sigma(\xR)$.
\end{theorem}
The latter result is proved in the last section. We conclude this introduction by fixing some notations.

\begin{notation}\label{N:21}
\begin{enumerate}[i)]
\item We denote by $f_x$ the spatial derivative of $f$ . 

\item $A \lesssim B$ means that there is $C>0$, depending only on fixed quantities, such that $A \leq C B$. 

\item Given $g=g(\alpha,x)$ and $Y$ a space of functions depending only on $x$, 
the notation 
$\lA g\rA_Y$ is a compact notation for $\alpha \mapsto \lA g(\alpha,\cdot)\rA_{Y}$.   
\end{enumerate}
\end{notation}

\section{Preliminaries}\label{S:preliminaries}

In this section, we recall or prove various 
results about Besov spaces which we will need throughout the article. 
We use the definition of these spaces originally given 
by Besov in~\cite{Besov}, using integrability properties of finite differences.

Given a real number $\alpha$, the finite difference operators $\delta_\alpha$ and $s_\alpha$ are defined by:
\begin{align*}
&\delta_{\alpha} f(x)=f(x)-f(x-\alpha),\\
&s_\alpha f(x)=2f(x)-f(x-\alpha)-f(x+\alpha).
\end{align*}
\begin{definition}
Consider three real numbers $(p,q,s)$ in $[1,\infty]^2 \times (0,2)$. 
The homogeneous Besov space $\dot B^{s}_{p,q}(\xR)$ consists of those 
tempered distributions $f$
whose Fourier transform is integrable near the origin and such that 
the following quantity $\lA f\rA_{\dot{B}^s_{p,q}}$ is finite: 
\begin{alignat}{2}
\lA f \rA_{\dot B^{s}_{p,q}} &=
\left\Vert \frac{\Vert  \delta_{\alpha} f\Vert_{L^{p}(\xR,\dx)}}{ \vert{\alpha}\vert^{s} } \right\Vert_{L^{q}(\mathbb R, \vert \alpha \vert^{-1} \dalpha)}
\label{Bs<1}\qquad &&\text{for }s\in(0,1),\\
\lA f \rA_{\dot B^{s}_{p,q}} &=   
\left\Vert \frac{\Vert s_\alpha f\Vert_{L^{p}(\xR;\dx)}}{ \vert{\alpha}\vert^{s} } \right\Vert_{L^{q}(\mathbb R, \vert \alpha \vert^{-1} \dalpha)}
\qquad &&\text{for }s\in[1,2).\label{Besov2}
\end{alignat}
\end{definition}
We refer the reader 
to the book of Peetre~\cite[chapter 8]{Peetre}  for 
the equivalence between these definitions and 
the one in terms of Littlewood-Paley decomposition (see also \cite[Prop.\ 9]{BourdaudMeyer} 
or~\cite[Theorems~$2.36$, $2.37$]{BCD} for the case $s\in (0,1]$). 

In this paper, we use only Besov spaces of the form
$$
\dot{B}^s_{2,2}(\xR),\quad \dot{B}^s_{\infty,2}(\xR),\quad \dot{B}^s_{2,1}(\xR).
$$
We will make extensive use of the fact that
$\lA \cdot\rA_{\dot H^{s}}$ and $\lA \cdot\rA_{\dot{B}^s_{2,2}}$ are equivalent for $s\in (0,2)$. 
Moreover, for $s\in (0,1)$,
\be\label{SE0}
\lA u\rA_{\dot H^{s}}^2=\frac{1}{4\pi c(s)}\lA u\rA_{\dot{B}^s_{2,2}}^2
\quad\text{with}\quad c(s)=\int_\xR \frac{1-\cos(t)}{\la t\ra^{1+2s}}\dt.
\ee 
We will also make extensive use of the fact that, for all $s$ in $(0,2)$,
\be\label{SE}
\dot H^{s+\mez} (\xR)\hookrightarrow \dot B^{s}_{\infty,2}(\mathbb R).
\ee
We will also use the following
\begin{lemma}\label{L:B21}
For any $s\in (0,1)$ and any $\delta>0$ such that $[s-\delta,s+\delta]\subset (0,1)$, 
$$
\dot{H}^{s-\delta}(\xR)\cap \dot{H}^{s+\delta}(\xR)\hookrightarrow \dot{B}^{s}_{2,1}(\xR).
$$
\end{lemma}
\begin{proof}
We have
\begin{align*}
\int_{\la\alpha\ra\le 1} 
 \frac{\Vert \delta_\alpha f\Vert_{L^{2}(\xR;\dx)}}{ \vert{\alpha}\vert^{s} } \frac{\dalpha}{\vert \alpha \vert}
&=\int_{\la\alpha\ra\le 1} 
\la\alpha\ra^\delta \frac{\Vert \delta_\alpha f\Vert_{L^{2}(\xR;\dx)}}{ \vert{\alpha}\vert^{s+\delta} } \frac{\dalpha}{\vert \alpha \vert}\\
&\le\left(\int_{\la\alpha\ra\le 1} 
\la\alpha\ra^{2\delta}\frac{\dalpha}{\vert \alpha \vert}\right)^\mez
\left(\int_{\xR}\frac{\Vert \delta_\alpha f\Vert_{L^{2}(\xR;\dx)}^2}{ \vert{\alpha}\vert^{2(s+\delta)} } 
\frac{\dalpha}{\vert \alpha \vert}\right)^\mez\\
&\le C(\delta)\lA f\rA_{\dot{B}^s_{2,2}}=C(\delta,s)\lA f\rA_{\dot{H}^{s+\delta}},
\end{align*}
and similarly
$$
\int_{\la\alpha\ra\ge 1} 
 \frac{\Vert \delta_\alpha f\Vert_{L^{2}(\xR;\dx)}}{ \vert{\alpha}\vert^{s} } \frac{\dalpha}{\vert \alpha \vert}
\le C'(\delta,s)\lA f\rA_{\dot{H}^{s-\delta}},
$$
which gives the  result.
\end{proof}
As an example of properties which are very simple to prove using 
the definition of Besov semi-norms in terms of finite differences, 
let us prove the first point in Theorem~\ref{T1}. Recall that, by notation,
$$
\mathcal{T}(f)g=-\frac{1}{\pi}\int \Delta_\alpha g_x \ \frac{(\Delta_\alpha f)^2}{1+(\Delta_\alpha f)^2}\dalpha,
$$
where $g_x\defn \partial_x g$.

\begin{prop}\label{P:continuity}
\begin{enumerate}[i)]
\item\label{Prop:low1} For all $f$ in $\dot{H}^{1}(\xR)$ and all $g$ in $\dot{H}^{\tdm}(\xR)$, 
the function
$$
\alpha\mapsto \Delta_\alpha g_x \ \frac{(\Delta_\alpha f)^2}{1+(\Delta_\alpha f)^2}
$$
belongs to $L^1_\alpha(\xR;L^2_x(\xR))$. Consequently, $\mathcal{T}(f)g$ belongs to $L^2(\xR)$. 
Moreover, there is a constant $C$ such that
\be\label{nTL2}
\lA \mathcal{T}(f)g\rA_{L^2}\le C \lA f\rA_{\dot{H}^1}\lA g\rA_{\dot{H}^\tdm}.
\ee

\item \label{Prop:low3} For all $\delta\in [0,1/2)$, there exists a constant $C>0$ such that, 
for all functions $f_1,f_2$ in $\dot{H}^{1-\delta}(\xR)\cap \dot{H}^{\tdm+\delta}(\xR)$, 
$$
\lA (\mathcal{T}(f_1)-\mathcal{T}(f_2))f_2\rA_{L^2}
\le C \lA f_1-f_2\rA_{\dot{H}^{1-\delta}}\lA f_2\rA_{\dot{H}^{\tdm+\delta}}.
$$

\item \label{Prop:low2} The map $f\mapsto \mathcal{T}(f)f$ is locally Lipschitz from 
$\dot{H}^1(\xR)\cap \dot{H}^{\tdm}(\xR)$ to $L^2(\xR)$.
\end{enumerate}
\end{prop}
\begin{proof}
$\ref{Prop:low1})$ Since
$$
\lA \Delta_\alpha g_x  \frac{(\Delta_\alpha f)^2}{1+(\Delta_\alpha f)^2}
\rA_{L^2}\le \lA \Delta_\alpha g_x\rA_{L^2}\lA \Delta_\alpha f\rA_{L^\infty}
=\frac{\lA \delta_\alpha g_x\rA_{L^2}}{\la \alpha\ra}\frac{\lA \delta_\alpha f\rA_{L^\infty}}{\la\alpha\ra},
$$
by using the Cauchy-Schwarz inequality and the definition~\e{Bs<1} of 
the Besov semi-norms one finds that
\begin{align*}
\lA \mathcal{T}(f)g\rA_{L^2}
&\le \frac{1}{\pi}\int \frac{\lA \delta_\alpha g_x\rA_{L^2}}{\la \alpha\ra}
\frac{\lA \delta_\alpha f\rA_{L^\infty}}{\la\alpha\ra}
\dalpha\\
&\le \frac{1}{\pi}\int 
\frac{\lA \delta_\alpha g_x\rA_{L^2}}{\la \alpha\ra^{1/2}}
\frac{\lA \delta_\alpha f\rA_{L^\infty}}{\la\alpha\ra^{1/2}}
\frac{\dalpha}{\la\alpha\ra}\\
&\le \frac{1}{\pi}
\left(\int \frac{\lA \delta_\alpha g_x\rA_{L^2}^2}{\la \alpha\ra}\frac{\dalpha}{\la\alpha\ra}
\right)^\mez\left(\int\frac{\lA \delta_\alpha f\rA_{L^\infty}^2}{\la\alpha\ra}
\frac{\dalpha}{\la\alpha\ra}\right)^\mez\\
&\le \frac{1}{\pi}\lA g_x\rA_{\dot{B}^{\mez}_{2,2}}\lA f\rA_{\dot{B}^{\mez}_{\infty,2}}.
\end{align*}
Recalling that $\lA \cdot\rA_{\dot{B}^{\mez}_{2,2}}$ and 
$\lA \cdot\rA_{\dot{H}^{\mez}}$ are equivalent semi-norms, 
and using the Sobolev embedding~\e{SE}, we have
$$
\lA g_x\rA_{\dot{B}^{\mez}_{2,2}}\les \lA g\rA_{\dot{H}^{\tdm}},\quad 
\lA f\rA_{\dot{B}^{\mez}_{\infty,2}}\les \lA f\rA_{\dot{H}^1},
$$
and hence we obtain the wanted inequality~\e{nTL2}.

$\ref{Prop:low3})$ 
Write that
$$
(\mathcal{T}(f_1)-\mathcal{T}(f_2))f_2=
-\frac{1}{\pi}\int \Delta_\alpha f_{2x} \Delta_\alpha (f_1-f_2) M(\alpha,x)\dalpha
$$
where
$$
M(\alpha,x)=\frac{(\Delta_\alpha f_1)+\Delta_\alpha f_2}{(1+(\Delta_\alpha f_1)^2)(1+(\Delta_\alpha f_2)^2)}.
$$
Since $\la M(\alpha,x)\ra\le 1$, by repeating similar arguments to those used in the first part 
(balancing 
the powers of $\alpha$ in a different way), we get
\begin{align*}
\lA (\mathcal{T}(f_1)-\mathcal{T}(f_2))f_2\rA_{L^2}
&\le \frac{1}{\pi}\int \frac{\lA \delta_\alpha f_{2x}\rA_{L^2}}{\la \alpha\ra}
\frac{\lA \delta_\alpha (f_1-f_2)\rA_{L^\infty}}{\la\alpha\ra}
\dalpha\\
&\le \frac{1}{\pi}\int \frac{\lA \delta_\alpha f_{2x}\rA_{L^2}}{\la \alpha\ra^{1/2+\delta}}\frac{\lA \delta_\alpha (f_1-f_2)\rA_{L^\infty}}{\la\alpha\ra^{1/2-\delta}}
\frac{\dalpha}{\la\alpha\ra}\\
&\le \frac{1}{\pi}\lA f_{2x}\rA_{\dot{B}^{\mez+\delta}_{2,2}}\lA f_1-f_2\rA_{\dot{B}^{\mez-\delta}_{\infty,2}}.
\end{align*}
which implies
$$
\lA (\mathcal{T}(f_1)-\mathcal{T}(f_2))f_2\rA_{L^2}\le C \lA f_1-f_2\rA_{\dot{H}^{1-\delta}}\lA f_2\rA_{\dot{H}^{\tdm+\delta}}.
$$

$\ref{Prop:low2})$ 
Consider $f_1$ and $f_2$ in $\dot{H}^1(\xR)\cap \dot{H}^{\tdm}(\xR)$. 
Then
$$
\mathcal{T}(f_1)f_1-\mathcal{T}(f_2)f_2=\mathcal{T}(f_1)(f_1-f_2)+(\mathcal{T}(f_1)-\mathcal{T}(f_2))f_2.
$$
Then~\e{nTL2} implies that the $L^2$-norm of the 
first term is bounded by
$$
C\lA f_1\rA_{\dot{H}^1}\lA f_1-f_2\rA_{\dot{H}^{\tdm}}.
$$ 
We estimate the second term 
by using $\ref{Prop:low3})$ applied with $\delta=0$. It follows that
$$
\lA \mathcal{T}(f_1)f_1-\mathcal{T}(f_2)f_2\rA_{L^2}\les 
\big(\lA f_1\rA_{\dot{H}^1\cap\dot{H}^{\tdm}}
+\lA f_2\rA_{\dot{H}^1\cap\dot{H}^{\tdm}}\big) \lA f_1-f_2\rA_{\dot{H}^1\cap\dot{H}^{\tdm}},
$$
which completes the proof.
\end{proof}
We gather in the following proposition 
the nonlinear estimates which will be needed.

\begin{prop}\label{P:embeddings}
\begin{enumerate}[i)]
\item \label{nonlinear-i}
Let $s\in (0,1)$, then $L^\infty(\xR)\cap \dot{H}^{s}(\xR)$ is an algebra. 
Moreover, for 
all $u,v$ in $L^\infty(\xR)\cap \dot{H}^s(\xR)$,
\be\label{product}
\lA uv\rA_{\dot{H}^{s}}\le 2 \lA u\rA_{L^\infty}\lA v\rA_{\dot{H}^{s}}+2\lA v\rA_{L^\infty}\lA u\rA_{\dot{H}^{s}}.
\ee

\item \label{nonlinear-ii}  Consider a $C^\infty$ function $F\colon \xR\rightarrow \xR$ satisfying
$$
\forall (x,y)\in \xR^2,\qquad \la F(x)-F(y)\ra \le K \la x-y\ra.
$$
Then, for all $s\in (0,1)$ and all $u \in \dot{H}^{s}(\xR)$, 
one has $F(u)\in \dot{H}^{s}(\xR)$ together with the estimate
\be\label{comp}
\lA F(u)\rA_{\dot{H}^s}\le K\lA u\rA_{\dot{H}^s}.
\ee

\item \label{nonlinear-iii}Consider a $C^\infty$ function $F\colon \xR\rightarrow \xR$ and 
a real number $\sigma$ in $(1,2)$. Then, there exists a 
non-decreasing function $\mathcal{F}\colon \xR\rightarrow \xR$ such that, for 
all $u \in \dot{H}^{\sigma-1}(\xR)\cap \dot{H}^{\sigma}(\xR)$ 
one has $F(u)\in \dot{H}^{\sigma}(\xR)$ together with the estimate
\be\label{comp2}
\lA F(u)\rA_{\dot{H}^\sigma}\le \mathcal{F}(\lA u\rA_{L^\infty})
\Big(\lA u\rA_{\dot{H}^{\sigma-1}}+\lA u\rA_{\dot{H}^{\sigma}}\Big).
\ee
\end{enumerate}
\end{prop}
\begin{remark}
$i)$ 
The inequality~\e{product} is the classical Kato--Ponce estimate~(\cite{KP}). 
We will use it only when $0<s<1$, for which one has a straightforward proof (see below). 

$ii)$ Statement $\ref{nonlinear-ii})$ is also elementary and classical (see~\cite{BourdaudMeyer}). 
Notice that~\e{comp2} is a sub-linear estimate, which means that the 
constant $K$ depends only on $F$ and not on $u$ (which is false in general for $s>1$). 

$iii)$ The usual estimate for composition implies that 
$$
\lA F(u)\rA_{\dot{H}^\sigma}\le \mathcal{F}(\lA u\rA_{L^\infty})\lA u\rA_{L^2\cap \dot{H}^{\sigma}}.
$$
The bound \e{comp2} improves the latter estimate 
in that one requires less control of the low frequency component. 
This will play a role in the proof of Lemmas~\ref{Lemma2.2} and~\ref{Lemma2.3}.
\end{remark}
\begin{proof}
$\ref{nonlinear-i})$ 
Since $\delta_{\alpha}(uv)=u \delta_\alpha v +(\tau_\alpha v)\delta_\alpha u$ 
where $\tau_\alpha v(x)=v(x-\alpha)$, we have
$$
\lA \delta_\alpha (uv)\rA_{L^2}\le
\lA u\rA_{L^\infty}\lA \delta_\alpha v\rA_{L^2}+\lA v\rA_{L^\infty}\lA \delta_\alpha u\rA_{L^2}.
$$
Directly from the definition~\e{Bs<1}, we deduce that
$$
\lA uv\rA_{\dot{B}^s_{2,2}}\le 2\lA u\rA_{L^\infty}\lA v\rA_{\dot{B}^s_{2,2}}+2\lA v\rA_{L^\infty}\lA  u\rA_{\dot{B}^s_{2,2}}.
$$
This implies \e{product} by virtue of the identity~\e{SE0} 
on the equivalence of $\lA \cdot\rA_{\dot{H}^s}$ 
and $\lA \cdot\rA_{\dot{B}^s_{2,2}}$.

$\ref{nonlinear-ii})$ Similarly, the inequality \e{comp} follows directly from the fact that 
$$
\lA \delta_\alpha F(u)\rA_{L^2} \le K \lA \delta_\alpha u\rA_{L^2}.
$$

$\ref{nonlinear-iii})$ 
We adapt the classical proof 
of the composition rule in nonhomogeneous Sobolev spaces, 
which is based on the Littlewood-Paley decomposition. 
Namely, choose a function 
$\Phi\in C^\infty_0(\{\xi;\la\xi\ra < 1\})$ which is equal to $1$ when $|\xi|\le 1/2$ and set 
$\phi(\xi)=\Phi(\xi/2)-\Phi(\xi)$ which is supported in the annulus $\{\xi;1/2\le |\xi|\le 2\}$. 
Then, for all $\xi\in\xR$, one has $\Phi(\xi)+\sum_{j\in \xN}\phi(2^{-j}\xi)=1$, 
which one can use to decompose tempered distribution. 
For $u\in \mathcal{S}'(\xR)$, we set $\Delta_{-1}u=\mathcal{F}^{-1}(\Phi(\xi)\widehat{u})$ 
and 
$\Delta_j u=\mathcal{F}^{-1}(\phi(2^{-j}\xi)\widehat{u})$ for $j\in\xN$. 
We also use the notation $S_j u =\sum_{-1\le p\le j-1} \Delta_pu$ for $j\ge 0$ (so that $S_0u=\Delta_{-1}u=\Phi(D_x)u$). 

The classical proof (see~\cite{AlinhacGerard,BCD,PGLR}) of the composition rule consists in 
splitting $F(u)$ as 
\begin{align*}
F(u)&=F(S_0 u)+F(S_1 u)-F(S_0u)+\cdots+F(S_{j+1}u)-F(S_ju)+\cdots\\
&=F(S_0u)+\sum_{j\in\xN} m_j \Delta_j u
\qquad \text{with}\quad m_j=\int_0^1 F'(S_ju+y\Delta_j u)\dy\\
&=F(S_0u)+\sum_{j\in\xN} m_j \Delta_j \tilde{u} \qquad \text{with}\quad\tilde{u}=u-\Phi(2D_x)u,
\end{align*}
where we used $\Delta_j \circ \Phi(2D_x)=0$ for $j\ge 0$. 
Then, the Meyer's multiplier 
lemma (see~\cite[Theorem 2]{MeyerBourbaki} or \cite[Lemma 2.2]{AlinhacGerard}) implies that 
$$
\Big\lVert \sum_{j\ge 0} m_j \Delta_j \tilde{u}\Big\rVert_{H^\sigma}
\le \mathcal{F}(\lA u\rA_{L^\infty})\lA \tilde{u}\rA_{H^\sigma}, 
$$
where, to clarify notations, we insist on the fact that above $H^\sigma$ is the nonhomogeneous Sobolev space. 
Since $\lA \tilde{u}\rA_{H^\sigma}\le \lA u\rA_{\dot{H}^\sigma}$, we see that 
the contribution of $\sum m_j \Delta_j \tilde{u}$ 
is bounded by the right-hand side of \e{comp2}. 
This shows that the only difficulty is to estimate the low frequency component $F(S_0u)$. 
We claim that
\be\label{ClaimLF}
\lA F(S_0u)\rA_{\dot{H}^\sigma}\le \mathcal{F}(\lA u\rA_{L^\infty})\lA u\rA_{\dot{H}^{\sigma-1}}.
\ee
To see this, we start with
$$
\lA F(S_0u)\rA_{\dot{H}^\sigma}
=\lA \partial_x \big(F(S_0u)\big)\rA_{\dot{H}^{\sigma-1}}
=\lA F'(S_0u)\partial_x S_0u \rA_{\dot{H}^{\sigma-1}},
$$
and then use the product rule~\e{product} with $s=\sigma-1\in (0,1)$,
\begin{align*}
\lA F'(S_0u)\partial_x S_0u \rA_{\dot{H}^{\sigma-1}}
&\le 
2\lA F'(S_0u)\rA_{L^\infty} \lA \partial_x S_0u \rA_{\dot{H}^{\sigma-1}}\\
&\quad+2\lA F'(S_0u)\rA_{\dot{H}^{\sigma-1}}\lA \partial_x S_0u \rA_{L^\infty}.
\end{align*}
Since $\la \xi\Phi(\xi)\ra\le 1$ one has the obvious inequality
$$
\lA \partial_x S_0u \rA_{\dot{H}^{\sigma-1}}\le 
\lA u \rA_{\dot{H}^{\sigma-1}}.
$$
On the other hand, since the support of the Fourier 
transform of $S_0u$ is included in the ball of center $0$ and radius $1$, it follows from the 
Bernstein's inequality that
$$
\lA S_0u \rA_{L^\infty}\le C_1\lA u\rA_{L^\infty},\qquad 
\lA \partial_x S_0u \rA_{L^\infty}\le C_2 \lA u\rA_{L^\infty}.
$$
The first estimate above also implies that
$$
\lA F'(S_0u)\rA_{L^\infty}\le \mathcal{F}_1(\lA S_0 u\rA_{L^\infty})\le \mathcal{F}_2(\lA u\rA_{L^\infty})
$$
where $\mathcal{F}_1(r)= \sup_{y\in [-r,r]}\la F'(y)\ra$ and $\mathcal{F}_2(r)=\mathcal{F}_1(C_1r)$. 
It thus remains only to estimate $\lA F'(S_0u)\rA_{\dot{H}^{\sigma-1}}$. 
Notice that we may apply the composition rule 
given in statement~$\ref{nonlinear-ii})$ since the index $\sigma-1$ belongs to $(0,1)$ 
and since $F'$ is Lipschitz on an open set containing $S_0u(\xR)$. 
The composition rule \e{comp} implies that
$$
\lA F'(S_0u)\rA_{\dot{H}^{\sigma-1}}\le K \lA S_0u\rA_{\dot{H}^{\sigma-1}}\le K \lA u\rA_{\dot{H}^{\sigma-1}}
$$
with
$$
K=\sup_{[-2\lA S_0u\rA_{L^\infty},2\lA S_0u\rA_{L^\infty}]}\la F''\ra\le \mathcal{F}_3(\lA u\rA_{L^\infty}).
$$
This proves that the $\dot{H}^\sigma$-norm of $F(S_0u)$ satisfies \e{ClaimLF} and hence 
it is bounded by the right-hand side of \e{comp2}, which completes the proof of statement 
$\ref{nonlinear-iii})$.
\end{proof}

For later purposes, we prove the following commutator estimate with the Hilbert transform.

\begin{lemma}\label{L:Hilbert}
Let~$0<\theta<\nu<1$. There exists a constant~$K$ 
such that for all $f\in C^{0,\nu}(\xR)$, and all $u$ in the nonhomogeneous space 
$H^{-\theta}(\xR)$,
\begin{equation}\label{commutator:estimateusual}
\left\lVert\mathcal{H}(fu)-f\mathcal{H} u\right\rVert_{L^2}
\le K\left\lVert f\right\rVert_{C^{0,\nu}}
\left\lVert u\right\rVert_{H^{-\theta}}.
\end{equation}
\end{lemma}
\begin{proof}
We establish this estimate 
by using the para-differential calculus of Bony~\cite{Bony}. 
We use the Littlewood-Paley decomposition (see the proof of Proposition~\ref{P:embeddings}) 
and denote by~$T_f$ the operator of para-multiplication by~$f$, so that 
$$
T_f u=\sum_{j\ge 1}S_{j-1}(f)\Delta_j u.
$$
Denote by $f^{\flat}$ the multiplication operator $ u\mapsto fu$ and 
introduce $\mathcal{H}_0$, the Fourier multiplier with symbol 
$-i(1-\Phi(\xi))\xi/\la\xi\ra$ where $\Phi\in C^\infty_0(\xR)$ is such that 
$\Phi(\xi)=1$ on a neighborhood of the origin. With these notations, one can 
rewrite the commutator~$\left[ \mathcal{H}, f^{\flat} \right]$ as
\begin{align}
\left[ \mathcal{H}, f^{\flat} \right]
&=\left[\mathcal{H},T_{f}\right]+\mathcal{H}(f^{\flat}-T_{f})-(f^{\flat}-T_{f})\mathcal{H}\notag\\
&=\left[\mathcal{H}_0,T_{f}\right]+(\mathcal{H}-\mathcal{H}_0)T_{f}
-T_f(\mathcal{H}-\mathcal{H}_0)+\mathcal{H}(f^{\flat}-T_{f})-(f^{\flat}-T_{f})\mathcal{H}.\label{rhs:para}
\end{align}
Notice that $\mathcal{H}-\mathcal{H}_0$ is a smoothing operator (that is an operator bounded 
from $H^\sigma$ to $H^{\sigma+t}$ for any real numbers $\sigma,t\in\mathbb{R}$). 
We then use two classical estimates for paradifferential operators (see~\cite{Bony,MePise}). 
Firstly, 
$$
\forall \sigma\in\mathbb{R},\quad 
\left\lVert T_f\right\rVert _{{H^{\sigma}}\rightarrow{H^{\sigma}}}\le 
c(\sigma)\left\lVert f\right\rVert_{L^\infty},
$$
so
\begin{align*}
\lA (\mathcal{H}-\mathcal{H}_0)T_{f}\rA_{H^{-\theta}\rightarrow L^2}&\le 
\lA \mathcal{H}-\mathcal{H}_0\rA_{H^{-\theta}\rightarrow L^2}\lA T_{f}\rA_{H^{-\theta}\rightarrow H^{-\theta}}\les 
\lA f\rA_{L^\infty},\\
\lA T_{f}(\mathcal{H}-\mathcal{H}_0)\rA_{H^{-\theta}\rightarrow L^2}&\le 
\lA T_{f}\rA_{L^2\rightarrow L^{2}}
\lA \mathcal{H}-\mathcal{H}_0\rA_{H^{-\theta}\rightarrow L^2}\les 
\lA f\rA_{L^\infty}.
\end{align*}
Secondly, since $\mathcal{H}_0$ is a Fourier multiplier 
whose symbol is a smooth function of order~$0$ (which means that its 
$k$th derivative is bounded 
by $C_k(1+|\xi|)^{-k}$), one has
$$
\forall \sigma\in\mathbb{R},\quad
\left\lVert\left[\mathcal{H}_0,T_f \right]\right\rVert _{{H^{\sigma}}\rightarrow{H^{\sigma+\nu}}}\le 
c(\nu,\sigma)\left\lVert f\right\rVert_{C^{0,\nu}}.
$$
In particular, 
$$
\left\lVert\left[\mathcal{H}_0,T_f \right]\right\rVert_{{H^{-\nu}}\rightarrow{L^2}}\les 
\left\lVert f\right\rVert_{C^{0,\nu}}.
$$
It remains only to estimate the last two terms in the right-hand side of \e{rhs:para}. We claim that
$$
\blA \mathcal{H}(f^{\flat}-T_{f})\brA_{H^{-\theta}\rightarrow L^2}+\blA (f^{\flat}-T_{f})\mathcal{H}
\brA_{H^{-\theta}\rightarrow L^2}\les\lA f\rA_{C^{0,\nu}}.
$$
Since $\mathcal{H}$ is bounded from $H^\sigma$ to itself for any $\sigma \in \xR$, it 
is enough to prove that
$$
\blA f^{\flat}-T_{f}\brA_{H^{-\theta}\rightarrow L^2}\les\lA f\rA_{C^{0,\nu}}.
$$
To do so, observe that
$$
fg-T_{f}g=\sum_{j,p\ge -1}(\Delta_j f)(\Delta_{p}g)-\sum_{-1\le j\le p-2}(\Delta_j f)(\Delta_{p}g)
=\sum_{j\ge -1} (S_{j+2}g)\Delta_{j}f.
$$
Then, using the Bernstein's inequality and the characterization of H\"older spaces in 
terms of Littlewood-Paley decomposition, it follows from the assumption $\theta<\nu$ that the series 
$\sum 2^{j(\theta-\nu)}$ converges, so
\begin{align*}
\lA fg-T_{f}g\rA_{L^2}&\le \sum \lA S_{j+2}g\rA_{L^2}\lA \Delta_{j}f\rA_{L^\infty}\\
&\les \sum 2^{\theta (j+2)}\lA g\rA_{H^{-\theta}}2^{-j\nu}\lA f\rA_{C^{0,\nu}}\les \lA g\rA_{H^{-\theta}}\lA f\rA_{C^{0,\nu}}.
\end{align*}

By combining the previous estimates, we have 
$\lA \left[ \mathcal{H}, f^{\flat} \right]\rA_{H^{-\theta}\rightarrow L^2}\les \lA f\rA_{C^{0,\nu}}$, 
which gives the result.
\end{proof}

\section{Commutator estimate}\label{S:commutator}

In this section we prove statement~$\ref{T1:commutator})$ in Theorem~\ref{T1}. 
Namely, we prove the following proposition.

\begin{prop}
Let $0<\eps<1/2$. There is a non-decreasing function $\mathcal{F}\colon \xR_+\rightarrow \xR_+$ 
such that, for all functions $f,g$ in $\cap_{\sigma\ge 1}\dot{H}^\sigma(\xR)$, 
\be\label{n3.2}
\lA\Lambda^{1+\eps}\mathcal{T}(f)g-\mathcal{T}(f)\Lambda^{1+\eps}g\rA_{L^2}
\le \mathcal{F}\big(\lA f\rA_{\dot{H}^1\cap \dot{H}^{\frac{3}{2}+\eps}}\big)
\lA f\rA_{\dot{H}^1\cap \dot{H}^{\frac{3}{2}+\eps}}\lA g \rA_{\dot H^{\tdm+\eps}\cap \dot H^2}.
\ee
\end{prop}
\begin{proof}
Recall that
$$
\mathcal{T}(f)g=-\frac{1}{\pi}\int \frac{(\Delta_\alpha f)^2}{1+(\Delta_\alpha f)^2}\Delta_\alpha g_x\dalpha.
$$
Since 
$$
\Lambda^{1+\eps}\Delta_\alpha g_x=\Delta_\alpha \big(\Lambda^{1+\eps}g_x\big),
$$
we have
$$
\Lambda^{1+\eps}\mathcal{T}(f)g =\mathcal{T}(f)\Lambda^{1+\eps}g+ R_1(f)g + R_2(f)g,
$$
where
\be
R_1(f)g=-\frac{1}{\pi}\int \big(\Delta_\alpha g_x\big) 
\Lambda^{1+\eps}  \bigg(\frac{(\Delta_\alpha f)^2}{1+(\Delta_\alpha f)^2}\bigg)
\dalpha\label{T2},
\ee
and
\begin{align*}
R_2(f)g&=
-\frac{1}{\pi}\int \big(\Lambda^{1+\eps} \big(u_\alpha v_\alpha\big)
-u_\alpha\Lambda^{1+\eps}v_\alpha - v_\alpha \Lambda^{1+\eps} u_\alpha\big)
\dalpha\quad \text{with}\\
u_\alpha&=\Delta_\alpha g_x,\qquad v_\alpha=\frac{(\Delta_\alpha f)^2}{1+(\Delta_\alpha f)^2}.
\end{align*}

We shall estimate these two terms separately. 
Classical results from 
paradifferential calculus (see \cite{Bony,CoMe,MePise}) would allow us to estimate them 
provided that we work in nonhomogeneous Sobolev spaces. 
In the homogeneous spaces we are considering, we shall see that one can 
derive similar results by using only elementary nonlinear estimates.

We begin with the study of $R_1(f)g$. 
\begin{lemma}\label{Lemma2.2}
There exists a non-decreasing function $\mathcal{F}\colon \xR_+\rightarrow \xR_+$ such that
$$
\lA R_1(f)g\rA_{L^2}\le \mathcal{F}\big(\lA f\rA_{\dot{H}^1\cap \dot{H}^{\tdm+\eps}}\big)
\lA f\rA_{\dot{H}^1\cap \dot{H}^{\tdm+\eps}}
\lA g \rA_{\dot H^{\tdm+\eps}\cap \dot H^2}.
$$
\end{lemma}
\begin{proof}
By definition
$$
R_1(f)g=-\frac{1}{\pi}\int \big(\Delta_\alpha g_x\big) 
\Lambda^{1+\eps}  \bigg(\frac{(\Delta_\alpha f)^2}{1+(\Delta_\alpha f)^2}\bigg)
\dalpha,
$$
so
\begin{equation}\label{T2-0}
\lA R_1(f)g\rA_{L^2} \les 
 \int \Vert \Delta_{\alpha}g_x\Vert_{L^{\infty}}
 \lA \frac{(\Delta_\alpha f)^2}{1+(\Delta_\alpha f)^2}\rA_{\dot{H}^{1+\eps}}\dalpha.
 \end{equation}
The Sobolev embedding $L^2(\xR)\cap \dot{H}^{\mez+\eps}(\xR)\hookrightarrow L^\infty(\xR)$ implies that, 
for all $\alpha$ in $\xR$,
\be\label{o1}
\sup_{x\in\xR}\la \Delta_\alpha f(x)\ra\le \sup_{x\in\xR}\la f_x(x)\ra\les 
\lA f_x\rA_{L^2\cap \dot{H}^{\mez+\eps}}=\lA f\rA_{\dot{H}^1\cap \dot{H}^{\tdm+\eps}},
\ee
so that the composition rule~\e{comp2} implies that
\be\label{T2-1}
\lA \frac{(\Delta_\alpha f)^2}{1+(\Delta_\alpha f)^2}\rA_{\dot{H}^{1+\eps}}
\le \mathcal{F}\big(\lA f\rA_{\dot{H}^1\cap \dot{H}^{\tdm+\eps}}\big)\Big(
\lA \Delta_\alpha f\rA_{\dot{H}^\eps}+\lA \Delta_\alpha f\rA_{\dot{H}^{1+\eps}}\Big).
\ee
We claim that we have the 
two following inequalities
\begin{align}
&\int \Vert \Delta_{\alpha}g_x\Vert_{L^{\infty}}
 \lA \Delta_\alpha f\rA_{\dot{H}^{\eps}}\dalpha\les \lA g \rA_{\dot H^{\tdm+\eps}}
 \lA f \rA_{\dot H^{1}}, \label{goalT2-1}\\
&\int \Vert \Delta_{\alpha}g_x\Vert_{L^{\infty}}
 \lA \Delta_\alpha f\rA_{\dot{H}^{1+\eps}}\dalpha\les \lA g\rA_{\dot H^2} \lA f \rA_{\dot{H}^{\tdm+\eps}}\label{goalT2-2}.
\end{align}
Let us prove \e{goalT2-1}. Directly from the definition of $\Delta_\alpha$, we have
\begin{align*}
\int \Vert \Delta_{\alpha}g_x\Vert_{L^{\infty}} 
\lA \Delta_\alpha f\rA_{\dot H^{\eps}}\dalpha
&=   \int \frac{\Vert \delta_{\alpha}g_{x}\Vert_{L^{\infty}}}{\vert \alpha \vert}
\frac{\Vert \delta_{\alpha}\Lambda^{\eps}f\Vert_{L^2}}{\vert \alpha \vert} \dalpha\\
&= \int \frac{\Vert \delta_{\alpha}g_{x}\Vert_{L^{\infty}}}{\vert \alpha \vert^{\eps}} \frac{\Vert \delta_{\alpha}\Lambda^{\eps}f\Vert_{L^2}}{\vert \alpha \vert^{1-\eps}} \frac{\dalpha}{\la\alpha\ra}.
\end{align*}
So, using the Cauchy-Schwarz inequality,
$$
\int \Vert \Delta_{\alpha}g_x\Vert_{L^{\infty}} 
\lA \Delta_\alpha f\rA_{\dot H^{\eps}}\dalpha\le 
\left(\int \frac{\Vert \delta_{\alpha}g_{x}\Vert_{L^{\infty}}^2}{\vert \alpha \vert^{2\eps}} \frac{\dalpha}{|\alpha|}\right)^{\mez}
\left(\int 
\frac{\Vert \delta_{\alpha}\Lambda^{\eps}f\Vert_{L^2}^2}{\vert \alpha \vert^{2(1-\eps)}} \frac{\dalpha}{\la\alpha\ra}\right)^\mez
$$
and hence, using the definition of Besov semi-norms (see~\e{Bs<1}), 
$$
\int \Vert \Delta_{\alpha}g_x\Vert_{L^{\infty}} 
\lA \Delta_\alpha f\rA_{\dot H^{\eps}}\dalpha\les 
\Vert g_x \Vert_{\dot B^{\eps}_{\infty,2}} 
\Vert \Lambda^{\eps} f \Vert_{\dot B^{1-\eps}_{2,2}}.
$$
By using~\e{SE0} and~\e{SE}, we obtain that
$$
\int \Vert \Delta_{\alpha}g_x\Vert_{L^{\infty}} 
\lA \Delta_\alpha f\rA_{\dot H^{\eps}}\dalpha\les 
\Vert g \Vert_{\dot H^{\tdm+\eps}} 
\Vert f \Vert_{\dot H^{1}},
$$
which is the first claim~\e{goalT2-1}. 
To prove the second claim~\e{goalT2-2}, 
we repeat the same arguments except that we 
balance the powers of $\alpha$ in a different way: 
\begin{align*}
\int \Vert \Delta_{\alpha}g_x\Vert_{L^{\infty}}
 \lA \Delta_\alpha f\rA_{\dot{H}^{1+\eps}}\dalpha&=
\int \frac{\Vert \delta_{\alpha}g_{x}\Vert_{L^{\infty}}}{\vert \alpha \vert^{1/2}} \frac{\Vert \delta_{\alpha}\Lambda^{1+\eps}f\Vert_{L^2}}{\vert \alpha \vert^{1/2}} \frac{\dalpha}{\la\alpha\ra}\\
&\le \left(\int \frac{\Vert \delta_{\alpha}g_{x}\Vert_{L^{\infty}}^2}{\vert \alpha \vert} \frac{\dalpha}{|\alpha|}\right)^{\mez}
\left(\int 
\frac{\Vert \delta_{\alpha}\Lambda^{1+\eps}f\Vert_{L^2}^2}{\vert \alpha \vert} \frac{\dalpha}{\la\alpha\ra}\right)^\mez\\
&\le \lA g_x \rA_{\dot B^{\mez}_{\infty,2}} \lA \Lambda^{1+\eps} f \rA_{\dot B^{\mez}_{2,2}}\\
&\les \lA g \rA_{\dot H^2} \lA f \rA_{\dot{H}^{\tdm+\eps}},
\end{align*}
which proves the claim~\e{goalT2-2}. Now, by combining the two claims 
\e{goalT2-1}, \e{goalT2-2} with \e{T2-0} and \e{T2-1}, we obtain that
$$
\lA R_1(f)g\rA_{L^2}\le \mathcal{F}\big(\lA f\rA_{\dot{H}^1\cap \dot{H}^{\tdm+\eps}}\big)
\lA f\rA_{\dot{H}^1\cap \dot{H}^{\tdm+\eps}}
\lA g \rA_{\dot H^{\tdm+\eps}\cap \dot H^2},
$$
which is the desired result.
\end{proof}

We now move to the second remainder term $R_2(f)g$. 
\begin{lemma}\label{Lemma2.3}
There exists a non-decreasing function $\mathcal{F}\colon \xR_+\rightarrow \xR_+$ such that
$$
\lA R_2(f)g\rA_{L^2}\le \mathcal{F}\big(\lA f\rA_{\dot{H}^1\cap \dot{H}^{\tdm+\eps}}\big)
\lA f\rA_{\dot{H}^1\cap \dot{H}^{\tdm+\eps}}
\lA g \rA_{\dot H^{\tdm+\eps}\cap \dot H^2}.
$$
\end{lemma}
\begin{proof}
We use the classical Kenig--Ponce--Vega commutator estimate
\begin{equation}\label{KPV}
\Vert \Lambda^{s}( uv) -u \Lambda^s v - v \Lambda^s u \Vert_{L^{r}} \leq C \lA \Lambda^{s_{1}} u \rA_{L^{p_1}} \lA \Lambda^{s_{2}} v \rA_{L^{p_2}}
\end{equation}
where $s=s_1+s_2$ and $1/r=1/p_1+1/p_2$. 
Kenig, Ponce and Vega considered the case $s<1$. 
Since, for our purpose we need $s>1$, 
we will use the recent improvement  by Li~\cite{Li} (see also D'Ancona~\cite{Dancona}) 
showing that \e{KPV} holds under 
the assumptions
$$
s=s_1+s_2\in (0,2),\quad s_j\in (0,1),
\quad \frac{1}{r}=\frac{1}{p_1}+\frac{1}{p_2},\quad 2\le p_j<\infty.
$$
With $p_1=4$, $p_2=4$, $r=2$, $s=1+\eps$, $s_1=\frac{3\eps}{2}$, $s_2=1-\frac{\eps}{2}$, this implies that
$$
\lA R_2(f)g\rA_{L^2}\les   \int \blA \Lambda^{\frac{3\eps}{2}}\Delta_{\alpha} g_{x} \brA_{L^{4}} \lA \Lambda^{1-\frac{\eps}{2}}\frac{(\Delta_\alpha f)^2}{1+(\Delta_\alpha f)^2} \rA_{L^4} \dalpha.
$$
We now use the Sobolev inequality
$$
\lA u\rA_{L^4}\les \blA \Lambda^{\frac{1}{4}} u\brA_{L^2},
$$
to obtain
$$
\lA R_2(f)g\rA_{L^2}\les   \int \blA \Lambda^{\frac14+\frac{3\eps}{2}}\Delta_{\alpha} g_{x} \brA_{L^{2}} \lA \Lambda^{\frac54-\frac{\eps}{2}}\frac{(\Delta_\alpha f)^2}{1+(\Delta_\alpha f)^2} \rA_{L^2} \dalpha.
$$
By combining the composition rule~\e{comp2} (applied with $\sigma=\frac54-\frac{\eps}2\in (1,2)$) and \e{o1}, 
we obtain that 
$$
\lA R_2(f)g\rA_{L^2}\les  \mathcal{F}(M)
\int \blA \Lambda^{\frac14+\frac{3\eps}{2}}\Delta_{\alpha} g_{x} \brA_{L^{2}} 
\Big(\blA \Lambda^{\frac14-\frac{\eps}{2}}\Delta_\alpha f \brA_{L^2}
+\blA \Lambda^{\frac54-\frac{\eps}{2}}\Delta_\alpha f \brA_{L^2}\Big) \dalpha
$$
where $M=\lA f\rA_{\dot{H}^1\cap \dot{H}^{\tdm+\eps}}$. 
We now proceed as in the previous proof. More precisely, 
we balance the powers of $\alpha$, use 
the Cauchy-Schwarz inequality, 
the definition of the Besov semi-norms~\e{Bs<1} and the 
Sobolev embedding to obtain that
\begin{align*}
&\int \blA \Lambda^{\frac14+\frac{3\eps}{2}}\Delta_{\alpha} g_{x} \brA_{L^{2}} 
\blA \Lambda^{\frac14-\frac{\eps}{2}}\Delta_\alpha f \brA_{L^2} \dalpha\\
&\qquad\qquad=\int \frac{\blA \Lambda^{\frac14+\frac{3\eps}{2}}\delta_{\alpha} g_{x} \brA_{L^2}}{\la\alpha\ra}
\frac{\blA \Lambda^{\frac14-\frac{\eps}{2}}\delta_\alpha f \brA_{L^2}}{\la\alpha\ra}\dalpha\\
&\qquad\qquad=\int \frac{\blA \Lambda^{\frac14+\frac{3\eps}{2}}\delta_{\alpha} g_x \brA_{L^2}}{\la\alpha\ra^{1/4-\eps/2}}
\frac{\blA \Lambda^{\frac14-\frac{\eps}{2}}\delta_\alpha f \brA_{L^2}}{\la\alpha\ra^{3/4+\eps/2}}\frac{\dalpha}{\la\alpha\ra}\\
&\qquad\qquad\le \lA \Lambda^{\frac14+\frac{3\eps}{2}}g_x\rA_{\dot{B}^{\frac14-\eps/2}_{2,2}}
\lA \Lambda^{\frac14-\frac{\eps}{2}} f\rA_{\dot{B}^{\frac34+\eps/2}_{2,2}}\\
&\qquad\qquad\les \lA g\rA_{\dot{H}^{\tdm+\eps}}
\lA f\rA_{\dot{H}^1}.
\end{align*}
One estimates the second term in a similar way. We begin by writing that
\begin{align*}
&\int \blA \Lambda^{\frac14+\frac{3\eps}{2}}\Delta_{\alpha} g_{x} \brA_{L^{2}} 
\blA \Lambda^{\frac54-\frac{\eps}{2}}\Delta_\alpha f \brA_{L^2} \dalpha\\
&\qquad\qquad=\int \frac{\blA \Lambda^{\frac14+\frac{3\eps}{2}}\delta_{\alpha} g_{x} \brA_{L^2}}{\la\alpha\ra}
\frac{\blA \Lambda^{\frac54-\frac{\eps}{2}}\delta_\alpha f \brA_{L^2}}{\la\alpha\ra}\dalpha\\
&\qquad\qquad=\int \frac{\blA \Lambda^{\frac14+\frac{3\eps}{2}}\delta_{\alpha} g_x \brA_{L^2}}{\la\alpha\ra^{3/4-3\eps/2}}
\frac{\blA \Lambda^{\frac54-\frac{\eps}{2}}\delta_\alpha f \brA_{L^2}}{\la\alpha\ra^{1/4+3\eps/2}}\frac{\dalpha}{\la\alpha\ra}.
\end{align*}
Since $\eps$ belongs to $(0,1/2)$ we have $3/4-3\eps/2>0$ and $1/4+3\eps/2<1$. Therefore one can use 
the definition \e{Bs<1} of the Besov semi-norms to deduce that
\begin{align*}
\int \blA \Lambda^{\frac14+\frac{3\eps}{2}}\Delta_{\alpha} g_{x} \brA_{L^{2}} 
\blA \Lambda^{\frac54-\frac{\eps}{2}}\Delta_\alpha f \brA_{L^2} \dalpha
&\le \blA \Lambda^{\frac{1}{4}+\frac{3\eps}{2}}g_x\brA_{\dot{B}^{\frac{3}{4}-\frac{3\eps}{2}}_{2,2}}
\blA \Lambda^{\frac{5}{4}-\frac{\eps}{2}} f\brA_{\dot{B}^{\frac{1}{4}+\frac{3\eps}{2}}_{2,2}}\\
&\les \lA g\rA_{\dot{H}^{2}}\lA f\rA_{\dot{H}^{\frac{3}{2}+\eps}}.
\end{align*}
By combining the above inequalities, we have proved that
$$
\lA R_2(f)g\rA_{L^2}\le \mathcal{F}\big(\lA f\rA_{\dot{H}^1\cap \dot{H}^{\tdm+\eps}}\big)
\lA f\rA_{\dot{H}^1\cap \dot{H}^{\tdm+\eps}}
\lA g\rA_{\dot{H}^{\tdm+\eps}\cap \dot{H}^{2}},
$$
which concludes the proof.
\end{proof}
This completes the proof of the proposition.
\end{proof}

\section{High frequency estimate}\label{S:high}

We now prove the second point of Theorem~\ref{T1} whose statement is recalled in the 
next proposition.

\begin{prop}\label{T1:prophigh}
For all $0<\nu<\eps<1/2$, there exists a positive constant $C>0$ 
such that, for all functions $f,g$ in $\cap_{\sigma\ge 1}\dot{H}^\sigma(\xR)$, 
$$
\mathcal{T}(f)g=\frac{f_x^2}{1+f_x^2}\Lambda g+V(f)\partial_x g+R(f,g)
$$
where
$$
\lA R(f,g)\rA_{L^2}\le C\lA f\rA_{\dot{H}^{\tdm+\eps}}\lA g\rA_{\dot{B}^{1-\eps}_{2,1}},
\qquad 
\lA V(f)\rA_{C^{0,\nu}}
\le C\lA f\rA_{\dot{H}^1\cap \dot{H}^{\frac{3}{2}+\eps}}^2.
$$
\end{prop}
We shall prove this proposition in this section by 
using a symmetrization argument which consists in replacing 
the finite differences $\delta_\alpha f(x)=f(x)-f(x-\alpha)$ 
by the symmetric finite differences $2f(x)-f(x-\alpha)-f(x+\alpha)$. 
To do so, it will be convenient to introduce a few notations. 

\begin{notation}\label{N:deltas}
Given a function $f=f(x)$ and a real number $\alpha$, we define 
the functions $\bar\delta_\alpha f$, $\bar\Delta_\alpha f$, $s_\alpha f$, 
$S_\alpha f$ and $D_\alpha f$ by:
\begin{align*}
&\bar\delta_\alpha f(x)=f(x)-f(x+\alpha),\\
&s_\alpha f(x)=\delta_\alpha f(x)+\bar\delta_\alpha f(x)=2f(x)-f(x-\alpha)-f(x+\alpha),
\end{align*}
and
\begin{align*}
&\bar\Delta_\alpha f(x)=\frac{f(x)-f(x+\alpha)}{\alpha},\\
&S_\alpha f(x)=\Delta_\alpha f(x)+\bar\Delta_\alpha f(x)=\frac{s_\alpha f(x)}{\alpha}=\frac{2 f(x)-f(x+\alpha)-f(x-\alpha)}{\alpha},\\
&D_{\alpha} f(x)=\Delta_\alpha f(x)-\bar\Delta_\alpha f(x)=\frac{f(x+\alpha)-f(x-\alpha)}{\alpha}.
\end{align*}
\end{notation}
\begin{lemma}\label{L:partialalpha}
One has
\begin{equation}\label{n:Dalpha}
 D_{\alpha} f =  2 f_x-\frac{1}{\alpha} \int_0^\alpha s_{\eta}f_{x}  \deta,
 \end{equation}
where $s_{\eta}f_{x}(x)=2f_x(x)-f_x(x+\eta)-f_x(x-\eta)$. Furthermore,
\be\label{n:paSD1}
\partial_{\alpha}(D_\alpha f)=-S_{\alpha}f_{x}+\frac{1}{\alpha^2}{\int_{0}^{\alpha} s_{\eta}f_{x} \deta},
\ee
and
\be\label{n:paSD2}
\partial_{\alpha}(S_\alpha f)=\bar\Delta_{\alpha} f_x-\Delta_{\alpha} f_x-\frac{S_{\alpha}f}{\alpha}.
\ee
\end{lemma}
\begin{proof}
The formula \e{n:Dalpha} can be verified by two direct calculations: one is
$$
\frac{1}{\alpha} \int_0^\alpha 2f_x(x)\deta=2f_x(x),
$$
and the other is
\begin{align*}
\frac{1}{\alpha} \int_0^\alpha( f_x(x-\eta)+f_x(x+\eta))\deta&=
\frac{1}{\alpha} \int_0^\alpha \partial_\eta (f(x+\eta)-f(x-\eta))\deta\\
&=\frac{1}{\alpha}(f(x+\alpha)-f(x-\alpha)).
\end{align*}
Now, the value for $ \partial_{\alpha}(D_\alpha f)$ in \e{n:paSD1} follows by differentiating \e{n:Dalpha}. 

The formula for $\partial_{\alpha}(S_\alpha f)$ follows 
from the definition of $S_\alpha f$ and the chain rule.
\end{proof}
Recall that
$$
\mathcal{T}(f)g=
-\frac{1}{\pi}\int \frac{(\Delta_\alpha f)^2}{1+(\Delta_\alpha f)^2}\Delta_\alpha g_x
\dalpha.
$$
The idea is to decompose the factor
$$
\frac{(\Delta_\alpha f)^2}{1+(\Delta_\alpha f)^2}
$$
into its even and odd components with respect to the variable $\alpha$. We
define
\begin{align}
\mathcal{E}(\alpha,\cdot)&=\mez \frac{(\Delta_\alpha f)^2}{1+(\Delta_\alpha f)^2}+\mez 
\frac{(\bar\Delta_\alpha f)^2}{1+(\bar\Delta_\alpha f)^2}
,\label{defi:E}\\
\mathcal{O}(\alpha,\cdot)&=\mez \frac{(\Delta_\alpha f)^2}{1+(\Delta_\alpha f)^2}
-\mez \frac{(\bar\Delta_\alpha f)^2}{1+(\bar\Delta_\alpha f)^2}
,\label{defi:O}
\end{align}
where the dots in the notations $\mathcal{E}(\alpha,\cdot)$ and $\mathcal{O}(\alpha,\cdot)$ are 
placeholders for the variable~$x$ (notice that 
$(\bar\Delta_\alpha f)^2=(\Delta_{-\alpha}f)^2$ and 
$\bar\Delta_\alpha f=-\Delta_{-\alpha}f$). 
Then,
$$
\mathcal{T}(f)g=- \frac{1}{\pi}
\int   \Delta_{\alpha} g_{x}\  \mathcal{E}(\alpha,\cdot) \dalpha  \\
-\frac{1}{\pi} \int  \Delta_{\alpha} g_{x}\ \mathcal{O}(\alpha,\cdot) \dalpha,
$$
and hence, since $\alpha \mapsto \mathcal{E}(\alpha,\cdot)$ 
is even, this yields $\mathcal{T}(f)g= \mathcal{T}_{e}(f)g+\mathcal{T}_{o}(f)g$ with
\begin{align*}
\mathcal{T}_{e}(f)g&=-\frac{1}{2\pi}\int   \big(\Delta_{\alpha} g_{x} - \bar\Delta_{\alpha} g_{x}\big) \
\mathcal{E}(\alpha,\cdot) \dalpha,  \\
\mathcal{T}_{o}(f)g&=-\frac{1}{\pi}
\int \Delta_{\alpha} g_{x}\ \mathcal{O}(\alpha,\cdot)\dalpha.
\end{align*}
The following result is the key point of the proof where 
we identify the main contribution of the nonlinearity.

\begin{prop}
There exists a constant $C$ such that
\be\label{T11}
\begin{aligned}
&\lA \mathcal{T}_{e}(f)g - \frac{f_x^2}{1+f_x^2}\Lambda g\rA_{L^2}
\le C\lA f\rA_{\dot{H}^{\tdm+\eps}}\lA g\rA_{\dot{B}^{1-\eps}_{2,1}},\\[1ex]
&\lA \mathcal{T}_{o}(f)g- V\partial_x g\rA_{L^2}\le C\lA f\rA_{\dot{H}^{\tdm+\eps}}\lA g\rA_{\dot{B}^{1-\eps}_{2,1}},
\end{aligned}
\ee
where
\be\label{T12}
V(x)=-\frac{1}{\pi}\int_\xR \frac{\mathcal{O}(\alpha,x)}{\alpha} \dalpha.
\ee
\end{prop}
\begin{proof}
$i)$ 
The main difficulty is to extract the elliptic component from $\mathcal{T}_{e}(f)g$. 
To uncover it, we shall perform 
an integration by parts in $\alpha$. 
The first key point is that
\begin{align*}
\Delta_{\alpha} g_x - \bar \Delta_{\alpha} g_x
&=\frac{g_x(\cdot+\alpha)-g_x(\cdot-\alpha)}{\alpha}\\
&=\frac{\partial_{\alpha}\left(g(\cdot+\alpha)+g(\cdot-\alpha)-2g(\cdot)\right)}{\alpha}\\
&=-\frac{\partial_{\alpha}(s_\alpha g)}{\alpha}.
\end{align*}
Consequently, directly from the definition of $\mathcal{T}_{e}(f)g$, 
by integrating by parts in $\alpha$, we obtain that
\be\label{nE1}
\begin{aligned}
\mathcal{T}_{e}(f)g
&=\frac{1}{2\pi}\int  \frac{\partial_{\alpha}( s_\alpha g)}{\alpha}
\mathcal{E}(\alpha,\cdot) \dalpha\\
&=\frac{1}{2\pi}\int \frac{s_\alpha g}{\alpha^2}  \mathcal{E}(\alpha,\cdot)\dalpha
-\frac{1}{2\pi}\int\frac{ s_\alpha g}{\alpha} \partial_\alpha \mathcal{E}(\alpha,\cdot)\dalpha.
\end{aligned}
\ee
We now have to estimate the coefficients $\mathcal{E}(\alpha,\cdot)$ 
and $\partial_\alpha\mathcal{E}(\alpha,\cdot)$.
\begin{lemma}
$i)$ We have 
\be\label{nQ1}
\mathcal{E}(\alpha,x)=\frac{f_x(x)^2}{1+f_x(x)^2}+Q(\alpha,x)
\ee
for some function $Q$ satisfying
\be\label{nQ2}
\la Q(\alpha,x)\ra\les \frac{\la s_\alpha f(x)\ra}{\la\alpha\ra}+ \la \frac{1}{\alpha} \int_0^\alpha s_{\eta}f_{x}(x) \deta\ra.
\ee
$ii)$ Furthermore,
\be\label{nE3}
\la \partial_\alpha \mathcal{E}(\alpha,x)\ra 
\le C \left\{
\frac{\la \bar\delta_{\alpha} f_x(x)\ra}{\la\alpha\ra}+
\frac{\la \delta_{\alpha} f_x(x)\ra}{\la\alpha\ra}+\frac{\la s_{\alpha}f(x)\ra}{\la\alpha\ra^2}+
\la \frac{1}{\alpha^2}{\int_{0}^{\alpha} s_{\eta}f_{x} (x)\deta}\ra\right\}
\ee
for some fixed constant $C$.
\end{lemma}
\begin{proof}
$i)$ We introduce the function
$$
F(a)=\frac{a^2}{1+a^2}.
$$
Then we have the identity \e{nQ1} with
\be\label{n700-1}
Q\defn \mez \left(F(\Delta_\alpha f)+F(\bar\Delta_\alpha f)\right)-F(f_x).
\ee
Since $F''$ is bounded, 
the Taylor formula implies that, 
for all $(a,b)\in \xR^2$,
$$
\la \mez (F(a)+F(b))-F\left(\frac{a+b}{2}\right)\ra\le \frac{\lA F''\rA_{L^\infty}}{8} \la a-b\ra^2.
$$
On the other hand, since 
$F$ is bounded, one has the obvious inequality
$$
\la \mez (F(a)+F(b))-F\left(\frac{a+b}{2}\right)\ra\le 2 \lA F\rA_{L^\infty}.
$$
By combining these two inequalities, we find that
$$
\la \mez (F(a)+F(b))-F\left(\frac{a+b}{2}\right)\ra\les \la a-b\ra.
$$
Since $F$ is even we have $F(b)=F(-b)$ and hence
$$
\la \mez (F(a)+F(b))
-F\left(\frac{a-b}{2}\right)\ra\les \la a+b\ra.
$$
We now apply this inequality with $a=\Delta_\alpha f$ and $b=\bar\Delta_\alpha f$. 
Since, by definition, $D_\alpha f=\Delta_\alpha f-\bar\Delta_\alpha f$, 
$S_\alpha f=\Delta_\alpha f+\bar\Delta_\alpha f$, we conclude that
\be\label{n700}
\la \mez (F(\Delta_\alpha f)+F(\bar\Delta_\alpha f))
-F\left(\mez D_\alpha f\right)\ra\les \la S_\alpha f\ra.
\ee
We now use the fact that $F$ 
is Lipschitz to infer from \e{n:Dalpha} that
\be\label{n701}
\la F\left(\mez D_\alpha f\right)
-F(f_x)\ra\les \la \frac{1}{\alpha}
\int_0^\alpha s_\eta f_x\deta\ra.
\ee
In light of \e{n700-1}, by using the triangle inequality, it follows from 
\e{n700} and 
\e{n701} that
$$
\la Q\ra\les \la S_\alpha f\ra+\la \frac{1}{\alpha}
\int_0^\alpha s_\eta f_x\deta\ra,
$$
which gives the result~\e{nQ2}.

$ii)$ Since
$$
\mathcal{E}(\alpha,\cdot)=\mez F(\Delta_\alpha f)+\mez F(\bar\Delta_\alpha f),
$$
and since $F'$ is bounded, the chain rule implies that 
$$
\la \partial_\alpha\mathcal{E}(\alpha,\cdot)\ra 
\les \la \partial_\alpha \Delta_\alpha f\ra+\la \partial_\alpha \bar\Delta_\alpha f\ra.
$$
By combining this estimate with the identities
$$
2\Delta_\alpha f=S_\alpha f+D_\alpha f, \quad 
2\bar\Delta_\alpha f=S_\alpha f-D_\alpha f,
$$
we deduce that 
$\la \partial_\alpha\mathcal{E}(\alpha,\cdot)\ra 
\les \la \partial_\alpha S_\alpha f\ra+\la \partial_\alpha D_\alpha f\ra$. 
Then the second estimate \e{nE3} follows from the 
values for $\partial_\alpha S_\alpha f$ and $\partial_\alpha D_\alpha f$ given by Lemma~\ref{L:partialalpha}.
\end{proof}

It follows directly from \e{nE1} and \e{nQ1}  that
\be\label{nE4}
\begin{aligned}
\mathcal{T}_{e}(f)g&=\frac{1}{2\pi}\frac{f_x^2}{1+f_x^2}
\int \frac{s_\alpha g}{\alpha^2} \dalpha\\
&\quad 
 +\frac{1}{2\pi}\int\frac{s_\alpha g}{\alpha} \left(\frac{Q(\alpha,\cdot)}{\alpha}-\partial_\alpha \mathcal{E}(\alpha,\cdot)\right)\dalpha.
\end{aligned}
\ee
Observe that
\begin{align*}
\int  \frac{s_{\alpha}g}{\alpha^2} \dalpha 
&=-\int s_{\alpha}g \partial_\alpha \left(\frac 1\alpha\right)\dalpha\\
&= \int \frac{\partial_\alpha s_{\alpha}g}{\alpha}\dalpha
=\int\frac{g_x(x-\alpha)-g_x(x+\alpha)}{\alpha}\dalpha\\
&=-2\int \Delta_\alpha g_x\dalpha= 
2\pi \Lambda g,
\end{align*}
where we used~\eqref{defiHilbertLambda}. So, the first term in the right-hand side of \e{nE4} is the wanted elliptic component
$$
\frac{f_x^2}{1+f_x^2}\Lambda g.
$$
To conclude the proof of the first statement in \e{T11}, 
it remains only to prove that the second term in the right-hand side 
of \e{nE4} is a remainder term. Putting for shortness
$$
I=\lA \int\frac{s_\alpha g}{\alpha} \left(\frac{Q(\alpha,\cdot)}{\alpha}-\partial_\alpha \mathcal{E}(\alpha,\cdot)\right)\dalpha\rA_{L^2},
$$
we will prove that 
\be\label{ndesired1}
I\les \lA f\rA_{\dot{H}^{\tdm+\eps}}\lA g\rA_{\dot{B}^{1-\eps}_{2,1}}.
\ee 
The $L^\infty$-norm of $\frac{Q(\alpha,\cdot)}{\alpha}-\partial_\alpha \mathcal{E}(\alpha,\cdot)$ is controlled 
from \e{nQ2} and~\e{nE3}. We have
$$
 I \les  I_1  +  I_2  + I_3 + I_4 
$$
with 
\be\label{n:I1234}
\begin{aligned}
I_1
&= \int \frac{\Vert s_{\alpha}g  \Vert_{L^{2}}}{\vert \alpha \vert^{1-\eps}}\frac{\Vert \bar\delta_{\alpha}f_x \Vert_{L^{\infty}}}{\vert \alpha \vert^{\eps}}  \frac{\dalpha}{\la \alpha\ra}, \\
I_2
&= \int \frac{\Vert s_{\alpha}g \Vert_{L^{2}}}{\vert \alpha \vert^{1-\eps}}\frac{\Vert \delta_{\alpha}f_x \Vert_{L^{\infty}}}{\vert \alpha \vert^{\eps}}  \frac{\dalpha}{\la \alpha\ra}, \\
I_3
&=  \int  \frac{\Vert s_{\alpha}g \Vert_{L^{2}}}{\vert \alpha \vert^{1-\eps}}
\frac{\Vert s_{\alpha}f \Vert_{L^{\infty}}}{\vert \alpha \vert^{1+\eps}}   \frac{\dalpha}{\la \alpha\ra}, \\
I_4 &=  \int  \frac{\Vert s_{\alpha}g \Vert_{L^{2}}}{\la\alpha\ra^{1-\eps}}
\frac{1}{\la\alpha\ra^{1+\eps}}   \la\int_0^\alpha \Vert s_{\eta}f_x \Vert_{L^{\infty}} \deta \ra\frac{\dalpha}{\la\alpha\ra},
\end{aligned}
\ee
where, as above, we have distributed the powers of $\la\alpha\ra$ in a balanced way. 
Using the Cauchy-Schwarz inequality and the definition~\e{Bs<1} of Besov semi-norms, 
it follows that
$$
I_1+I_2\le  \Vert g \Vert_{\dot B^{1-\eps}_{2,2}} \Vert f_{x} \Vert_{\dot B^{\eps}_{\infty,2}}. 
$$
and, similarly, it results from~\e{Besov2} that
$$
I_3\le  \Vert g \Vert_{\dot B^{1-\eps}_{2,2}} \Vert f \Vert_{\dot B^{1+\eps}_{\infty,2}}. 
$$
Consequently, the Sobolev embeddings
$$
\dot B^{1-\eps}_{2,1}\hookrightarrow \dot B^{1-\eps}_{2,2},\qquad
\dot H^{\tdm+\eps}(\xR)\hookrightarrow \dot B^{\eps}_{\infty,2}(\xR),
$$
imply that $I_1+I_2+I_3\les \lA g\rA_{\dot{B}^{1-\eps}_{2,1}}\lA f\rA_{\dot{H}^{\tdm+\eps}}$. 

To estimate $I_4$, the key point consists in using 
the Cauchy-Schwarz inequality to verify that
$$
\frac{1}{\la\alpha\ra^{1+\eps}}   \la\int_0^\alpha \Vert s_{\eta}f_x \Vert_{L^{\infty}} \deta \ra
\les \left(\int_0^\infty \frac{\Vert s_{\mu}f_x \Vert^{2}_{L^{\infty}}}{\mu^{2\eps}} \frac{\dmu}{\mu} \right)^{\mez}
$$
(notice that the variable $\eta$ above could be negative, while $\mu$ here is always positive). 
It follows from~\e{Bs<1} that 
$$
\left(\int_0^\infty \frac{\Vert s_{\mu}f_x \Vert^{2}_{L^{\infty}}}{\mu^{2\eps}} \frac{\dmu}{\mu} \right)^{\mez}
\les \lA f_x\rA_{\dot{B}^{\eps}_{\infty,2}}\les \lA f\rA_{\dot{H}^{\tdm+\eps}}, 
$$
we obtain, using again~\e{Bs<1} with $(p,q,s)=(2,1,1-\eps)$,
\be\label{n9}
I_4  \les  \lA f\rA_{\dot{H}^{\tdm+\eps}}
\int\frac{\Vert s_{\alpha}g \Vert_{L^{2}}}{\la\alpha\ra^{1-\eps}}\frac{\dalpha}{\la \alpha\ra}
\les\lA f\rA_{\dot{H}^{\tdm+\eps}}\lA g\rA_{\dot{B}^{1-\eps}_{2,1}}.
\ee
This completes the proof of \e{ndesired1} and hence the proof of the desired result~\e{T11}.

$ii)$ It remains to study $\mathcal{T}_{o}(f)g$. Recall that
$$\mathcal{T}_{o}(f)g=-\frac{1}{\pi}
\int \big( \Delta_{\alpha} g_{x}\big) \mathcal{O}(\alpha,\cdot)\dalpha,
$$
where $\mathcal{O}(\alpha,\cdot)$ is given by \e{defi:O}. 
By splitting the factor $\Delta_\alpha g_x$ into two parts
$$
\Delta_\alpha g_x(x)=\frac{g_x(x)}{\alpha}-\frac{g_x(x-\alpha)}{\alpha},
$$
we obtain at once that
$$
\mathcal{T}_{o}(f)g=V\partial_x g +B
$$
where $V$ is given by~\e{T12} and where the remainder $B$ is given by
\begin{equation*}
B(x)=\frac{1}{\pi} \int \frac{1}{\alpha}g_{x}(x-\alpha)\mathcal{O}(\alpha,x) \dalpha.
\end{equation*}
The analysis of $B$ is based on the observation that
$$
g_{x}(x-\alpha)=\partial_{\alpha}(g(x)-g(x-\alpha))=\partial_\alpha (\delta_\alpha g),
$$
which allows to integrate by parts in $\alpha$, to obtain
$$
B=\frac{1}{\pi} \int \frac{\delta_\alpha g }{\alpha}\left(\frac{1}{\alpha}\mathcal{O}(\alpha,\cdot)- \partial_\alpha  \mathcal{O}(\alpha,\cdot) \right)\dalpha.
$$
Consequently, by writing
$$
\lA B \rA_{L^2} \les  
\int  \frac{\Vert \delta_\alpha g \Vert_{L^{2}}}{\la \alpha\ra} 
\lA \frac{1}{\alpha}\mathcal{O}(\alpha,\cdot)- \partial_\alpha  \mathcal{O}(\alpha,\cdot)
\rA_{L^\infty}\frac{\dalpha}{\la\alpha\ra},
$$
we are back to the situation already treated in the first step. 
The estimate for $\partial_\alpha \mathcal{O}$ is proved by repeating the 
arguments used to prove the estimate~\e{nE3} for $\partial_\alpha \mathcal{E}$. 
To bound $\alpha^{-1}\mathcal{O}(\alpha,x)$, remembering the expression 
of $\mathcal{O}(\alpha,x)$ given by \e{defi:O}, it is sufficient to notice that
\be\label{esti:O}
\begin{aligned}
\la \frac{\mathcal{O}(\alpha,\cdot)}{\alpha}\ra
&=\frac{1}{2\la\alpha\ra} \la  \frac{(\Delta_\alpha f)^2}{1+(\Delta_\alpha f)^2}
-\frac{(\bar\Delta_\alpha f)^2}{1+(\bar\Delta_\alpha f)^2}\ra \\
&\le \frac{1}{2\la\alpha\ra}\la \frac{\Delta_\alpha f-\bar\Delta_\alpha f}{(1+(\Delta_\alpha f)^2)(1+(\bar\Delta_\alpha f)^2}
\ra
\la \Delta_\alpha f+\bar\Delta_\alpha f\ra\\
&\le \frac{\la S_\alpha f\ra}{\la\alpha\ra}\cdot
\end{aligned}
\ee
This gives that 
$\la \mathcal{O}(\alpha,x)\ra\les \la s_\alpha f(x)\ra/\la\alpha\ra^2$. 
Therefore, we obtain that the $L^\infty$-norm of $\frac{\mathcal{O}(\alpha,\cdot)}{\alpha}-\partial_\alpha \mathcal{O}(\alpha,\cdot)$ is estimated by the right-hand side of~\e{nE3}. 
Then we may repeat the arguments used in the proof of the first step to estimate $I$. 
We call the attention 
to the fact that, previously, in \e{n:I1234}, the expressions involved the more favorable 
symmetric differences $s_\alpha g$ instead of $\delta_\alpha g$. 
However, this is not important for our purpose since, to estimate $I_1,\ldots,I_4$, we used only the characterization of Besov norms valid for $0<s<1$, 
which involves only the finite differences  $\delta_\alpha f$. 
This proves that $\lA B \rA_{L^2}$ is controlled by the right-hand side of \e{ndesired1}, 
which implies that $B\sim 0$. 
\end{proof}

\begin{lemma} \label{hold}
Let $0<\nu<\eps<1/2$. 
There exists a positive constant $C>0$ such that
$$
\lA V\rA_{C^{0,\nu}}=\lA V\rA_{L^\infty}+\sup_{y\in \xR}\bigg(\frac{\la V(x+y)-V(x)\ra}{\la y\ra^\nu}\bigg)
\le C\lA f\rA_{\dot{H}^1\cap \dot{H}^{\frac{3}{2}+\eps}}^2.
$$
\end{lemma}
\begin{proof}
As already seen, we have
$$
V(x)=-\frac{1}{\pi}\int_\xR \frac{\mathcal{O}(\alpha,x)}{\alpha} \dalpha,
$$
where
$$
\mathcal{O}(\alpha,\cdot)=M_\alpha(f)S_\alpha f\quad\text{with}\quad 
M_\alpha(f)=\mez \frac{\Delta_\alpha f-\bar\Delta_\alpha f}{(1+(\bar\Delta_\alpha f)^2)(1+(\bar\Delta_\alpha f)^2}.
$$
Since $\la M_\alpha (f)\ra\le \la \Delta_\alpha f\ra+\la\bar\Delta_\alpha f\ra$, we obtain that
\begin{align*}
\lA V\rA_{L^\infty}&\le 2 \int \big(\lA \Delta_\alpha f\rA_{L^\infty}+\lA\bar\Delta_\alpha f\rA_{L^\infty}\big)\lA S_\alpha f\rA_{L^\infty}\frac{\dalpha}{\la\alpha\ra}\\
&\le \int \frac{\lA \delta_\alpha f\rA_{L^\infty}+\lA\bar\delta_\alpha f\rA_{L^\infty}}{\la\alpha\ra^{1-\eps}}
\frac{\lA s_\alpha f\rA_{L^\infty}}{\la \alpha\ra^{1+\eps}}\frac{\dalpha}{\la\alpha\ra}\\
&\le \lA f\rA_{\dot{B}^{1-\eps}_{\infty,2}}\lA f\rA_{\dot{B}^{1+\eps}_{\infty,2}}\les 
\lA f\rA_{\dot{H}^{1}\cap \dot{H}^{\tdm+\eps}}^2,
\end{align*}
where we used the Cauchy-Schwarz inequality, the definitions~\e{Bs<1} and~\e{Besov2} of the Besov semi-norms, and the Sobolev embedding.

We now have to estimate the H\"older-modulus of continuity of $V$. Given $y\in \xR$ and a function $u=u(x)$, 
we introduce the function $[u]_y$ defined by
$$
[u]_y(x)=\frac{u(x+y)-u(x)}{\la y\ra^{\nu}}\cdot
$$
We want to estimate the $L^\infty$-norm of $[V]_y$ uniformly in $y\in \xR$. 
Notice that
$$
[\mathcal{O}(\alpha,\cdot)]_y=M_\alpha(f)S_\alpha ([f]_y)+[M_\alpha (f)]_y \tau_y(S_\alpha f)
$$
where $\tau_y u(x)=u(x+y)$. The contribution of the first term is estimated as above: 
by setting $\delta=\eps-\nu>0$, we have
\begin{align*}
\int \lA M_\alpha(f)S_\alpha ([f]_y)\rA_{L^\infty}\dalpha
&\le 
\int \big(\lA \Delta_\alpha f\rA_{L^\infty}+\lA\bar\Delta_\alpha f\rA_{L^\infty}\big)
\lA S_\alpha [f]_y\rA_{L^\infty}\frac{\dalpha}{\la\alpha\ra}\\
&\le \int \frac{\lA \delta_\alpha f\rA_{L^\infty}+\lA\bar\delta_\alpha f\rA_{L^\infty}}{\la\alpha\ra^{1-\delta}}
\frac{\lA s_\alpha [f]_y\rA_{L^\infty}}{\la \alpha\ra^{1+\delta}}\frac{\dalpha}{\la\alpha\ra}\\
&\le \lA f\rA_{\dot{B}^{1-\delta}_{\infty,2}}\lA [f]_y\rA_{\dot{B}^{1+\delta}_{\infty,2}}.
\end{align*}
Now, using Plancherel theorem and the inequality $\la e^{iy\xi}-1\ra\le \la y\xi\ra^\nu$, we have
$$
\lA [f]_y\rA_{\dot{B}^{1+\delta}_{\infty,2}}\les \lA [f]_y\rA_{\dot{H}^{\tdm+\delta}}\les \lA f\rA_{\dot{H}^{\tdm+\delta+\nu}}
=\lA f\rA_{\dot{H}^{\tdm+\eps}},
$$
since $\delta+\nu=\eps$. 
On the other hand, since
$$
\la [M_\alpha (f)]_y\ra\les \la \Delta_\alpha [f]_y\ra+\la \bar\Delta_\alpha [f]_y\ra,
$$
by repeating the previous arguments, we get
\begin{align*}
\int \lA [M_\alpha (f)]_y \tau_y(S_\alpha f)\rA_{L^\infty}\dalpha
&\le \int \frac{\lA \delta_\alpha [f]_y\rA_{L^\infty}+\lA\bar\delta_\alpha [f]_y\rA_{L^\infty}}{\la\alpha\ra^{1-\nu}}
\frac{\lA s_\alpha f\rA_{L^\infty}}{\la \alpha\ra^{1+\nu}}\frac{\dalpha}{\la\alpha\ra}\\
&\le \lA [f]_y\rA_{\dot{B}^{1-\nu}_{\infty,2}}\lA f\rA_{\dot{B}^{1+\nu}_{\infty,2}}
\les \lA f\rA_{\dot{H}^{\tdm}}\lA f\rA_{\dot{H}^{\tdm+\nu}}.
\end{align*}
This concludes the proof of Lemma~\ref{hold} 
\end{proof}
This completes the proof of Theorem~\ref{T1}.

\section{Cauchy problem}

In this section we prove Theorem~\ref{T:Cauchy} about the Cauchy problem. 

We prove the uniqueness by estimating the difference of two solutions. 
With regards to the existence, we construct solutions to the Muskat equation 
as limits of solutions to a sequence of approximate nonlinear systems, 
following here \cite{ABZ1,ABZ3,LannesJAMS,LannesLivre}. 
We split the analysis in three parts.
\begin{enumerate}
\item Firstly, we prove that the Cauchy problem for these approximate 
systems are well-posed locally in time by means of an ODE argument.
\item Secondly, we use Theorem~\ref{T1} and an elementary $L^2$-estimate for 
the paralinearized equation to prove that 
the solutions of the later approximate systems are 
bounded in $C^0([0,T];\dot{H}^1(\xR)\cap\dot{H}^s(\xR))$ 
on a uniform time interval.
\item The third task consists in showing that these approximate solutions converge to a limit which is 
a solution of the Muskat equation. To do this, one cannot apply standard compactness results since 
the equation is non-local. Instead, we prove that 
the solutions form a Cauchy sequence in an appropriate space, by estimating 
the difference of two solutions.
\end{enumerate}

\subsection{Approximate systems}\label{S:approximate}
To define approximate systems, we use a 
version of Galerkin's method based on Friedrichs mollifiers. 
We find convenient to use smoothing operators 
which are projections and consider, for $n\in \xN\setminus\{0\}$, the operators $J_n$ defined by 
\begin{alignat*}{3}
\widehat{J_n u}(\xi)&=\hat{u}(\xi) \quad &&\text{for} &&\la \xi\ra\le n,\\
\widehat{J_n u}(\xi)&=0 \quad &&\text{for} &&\la \xi\ra> n.
\end{alignat*}
Notice that $J_n$ is a projection, $J_n^2=J_n$. This will allow us
 to simplify some technical arguments. 
 
Now we consider the following approximate Cauchy problems:
\begin{equation}\label{A3}
\left\{
\begin{aligned}
&\partial_t f+\Lambda f=J_n\big(\mathcal{T}(f)f\big),\\
& f\arrowvert_{t=0}=J_n f_{0}.
\end{aligned}
\right.
\end{equation} 
The following lemma states that this system has smooth local in time solutions.

\begin{lemma}
For all $f_0\in \dot{H}^{1}(\xR)$, and any $n\in \xN\setminus\{0\}$, the initial value 
problem~\e{A3} has a unique maximal solution, for some time $T_n>0$, 
of the form $f_n=J_nf_0+u_n$ where $u_n\in C^{1}([0,T_n[;H^{\infty}(\xR))$ 
is such that $u_n(0)=0$. Moreover, either 
\be\label{A4}
T_n=+\infty\qquad\text{or}\qquad \limsup_{t\rightarrow T_n} \lA u_n(t)\rA_{L^2}=+\infty.
\ee
\end{lemma}
\begin{proof}
We begin by studying an auxiliary system. Consider the following Cauchy problem
\begin{equation}\label{A3-twisted}
\left\{
\begin{aligned}
&\partial_t f+J_n\Lambda f=J_n\big(\mathcal{T}(J_n f)J_n f\big),\\
& f\arrowvert_{t=0}=J_n f_{0}.
\end{aligned}
\right.
\end{equation} 
Set $u=f-J_nf_0$. Then the Cauchy problem \eqref{A3-twisted} has the form
\begin{equation}\label{edo}
\partial_t u= F_n(u),\quad u\arrowvert_{t=0}=0,
\end{equation}
where 
$$
F_n(u)=-\Lambda J_n u-\Lambda J_n f_0+J_n\big(\mathcal{T}(J_n(f_0+u))J_n(f_0+u)\big)
$$
(we have used $J_n^2=J_n$ to simplify the expression of $F$). 
The operator $J_n$ is a smoothing operator: it is bounded from $\dot{H}^1(\xR)$ into $\dot{H}^\mu(\xR)$ for any $\mu\ge 1$, and from 
$L^2(\xR)$ into $H^\mu(\xR)$ for any $\mu\ge 0$. Consequently, 
if $u$ belongs to $L^{2}(\xR)$, then $J_n(f_0+u)$ belongs to $\dot{H}^\mu(\xR)$ for any 
$\mu\ge 1$. Thus, it follows from statement~$\ref{Prop:low1})$ in Proposition~\ref{P:continuity} and the assumption $f_0\in \dot{H}^1(\xR)$ 
that $F_n$ maps $L^2(\xR)$ into itself. This shows that~\eqref{edo} is in fact an ODE with values 
in a Banach space for any $n\in\xN\setminus\{0\}$. 
The key point is that statement~$\ref{Prop:low2})$ in Proposition~\ref{P:continuity} implies that 
the function $F_n$ is locally Lipschitz 
from $L^{2}(\xR)$ to itself. 
Consequently, the Cauchy-Lipschitz theorem gives 
the existence of a unique maximal solution~$u_n$ 
in~$C^{1}([0,T_n[;L^{2}(\xR))$. Then the function $f_n=J_nf_0+u_n$ is a solution to \e{A3-twisted}. 
Since $J_n^2=J_n$, we check that the function $(I-J_n)f_n$ solves
$$
\partial_t (I-J_n)f_n=0,\quad (I-J_n)f_n\arrowvert_{t=0}=0.
$$
This shows that $(I-J_n)f_n=0$, so $J_nf_n=f_n$. Consequently, 
the fact that $f_n$ solves \e{A3-twisted} implies 
that $f_n$ is also a solution to~\e{A3}. 

The alternative \e{A4} is a consequence of the usual 
continuation principle for ordinary differential equations. 
Eventually, integrating~\e{edo} in time and using the fact that $J_n$ is a smoothing operator, 
we obtain that $u_n$ belongs to $C^0([0,T_n[;H^{\infty}(\xR))$. Using again~\e{edo}, we conclude 
that $\partial_t u_n$ belong to $C^0([0,T_n[;H^{\infty}(\xR))$. 
\end{proof}

\subsection{A priori estimate for the approximate systems}

In this paragraph we prove two {\em a priori} estimates which will play a key role to prove uniform 
estimates for the solutions $(f_n)$ and also to estimate the differences between two such solutions. 
We begin with the following estimate in $L^2(\xR)\cap \dot{H}^s(\xR)$.

\begin{prop}\label{Prop:aprioriHs}
For all real number $s\in (3/2,2)$, there exists 
a positive constant $C>0$ and 
a non-decreasing function $\mathcal{F}\colon\xR\rightarrow\xR$ such that, any 
$n\in \xN\setminus\{0\}$, for any $T\in (0,T_n)$,  the norm 
$$
M_n(T)=\sup_{t\in [0,T]}\lA f_n(t)-f_0\rA_{L^2\cap \dot{H}^s}^2
$$
satisfies
\be\label{apriori}
\begin{aligned}
&M_n(T)+\frac{C}{1+K^2}\int_0^T\lA f_n(t)\rA_{\dot{H}^{s+\mez}}^2\dt\\
&\qquad\qquad\qquad\le (2+T)^2\lA f_0\rA_{\dot{H}^1\cap \dot{H}^s}^2
+T\mathcal{F}\big(\sup_{t\in [0,T]}\lA f_n(t)\rA_{\dot{H}^1\cap \dot{H}^s}^2\big),
\end{aligned}
\ee
where
\be\label{boundK}
K\defn \sup_{(t,x)\in [0,T]\times\xR}\la \partial_x f_n(t,x)\ra.
\ee
\end{prop}
\begin{proof}
Set $\mathcal{T}_n(f)=J_n\big(\mathcal{T}(f)f\big)$. 
We estimate the $\lA \cdot\rA_{L^2}$-norm 
and $\lA \cdot\rA_{\dot{H}^s}$-norm by different methods. 

\smallbreak

\noindent{\em First step : low-frequency estimate.} Since
\be\label{n1509}
\partial_t f_n+\Lambda f_n=J_n\mathcal{T}(f_n)f_n,\quad f_n\arrowvert_{t=0}=J_nf_0,
\ee
we have
$$
f_n(t)=\exp\left(-t\Lambda\right)J_nf_0+\int_0^t \exp\left(-(t-t')\Lambda\right)\mathcal{T}_n(f_n)(t')\dt',
$$
so
$$
f_n(t)-J_nf_0
=\left(\exp\left(-t\Lambda\right)-I\right)J_nf_0+\int_0^t \exp\left(-(t-t')\Lambda\right)\mathcal{T}_n(f_n)(t')\dt',
$$
where $I$ denotes the identity operator. Using the Fourier transform and Plancherel identity, one 
obtains immediately that
$$
\lA \left(\exp\left(-t\Lambda\right)-I\right)J_nf_0\rA_{L^2}\le \lA t\Lambda J_nf_0\rA_{L^2}
\le  T\lA f_0\rA_{\dot{H}^1}.
$$
On the other hand, 
$$
\lA \exp\left(-(t-t')\Lambda\right)\mathcal{T}_n(f_n)(t')\rA_{L^2}\le \lA \mathcal{T}_n(f_n)(t')\rA_{L^2}.
$$
Consequently, 
$$
\lA f_n(t)-J_nf_0\rA_{L^2}\le T\lA f_0\rA_{\dot{H}^1}+T\sup_{t'\in [0,T]}
\lA \mathcal{T}_n(f_n)(t')\rA_{L^2}.
$$
Now we want to replace the left hand side 
of the above inequality by $\lA f_n(t)-f_0\rA_{L^2}$. To do so, notice that, 
since the spectrum of $J_nf_0-f_0$ is contained in $\{\la\xi\ra\ge 1\}$, we have
$$
\lA J_nf_0-f_0\rA_{L^2}\le \lA f_0\rA_{\dot{H}^1}.
$$
By combining the above estimates, we deduce that
$$
\lA f_n(t)-f_0\rA_{L^2}\le (1+T)\lA f_0\rA_{\dot{H}^1}+T\sup_{t'\in [0,T]}
\lA \mathcal{T}_n(f_n)(t')\rA_{L^2}.
$$
Now, we estimate the $L^2$-norm of the nonlinearity $\mathcal{T}_n(f_n)$ by means of the 
first statement in Theorem~\ref{T1}. 
We conclude that, for $T<1$,
$$
\lA f_n(t)-f_0\rA_{L^2}^2\le  2(1+T)^2\lA f_0\rA_{\dot{H}^1}^2+C T^2\sup_{[0,T]}\lA f_n\rA_{\dot{H}^1}^2
\sup_{[0,T]}\lA f_n\rA_{\dot{H}^{\tdm}}^2.
$$
This is in turn estimated by the right side of \e{apriori}. This concludes the first step.

\bigbreak

\noindent{\em Second step : High frequency estimate.} 
Denote by $(\cdot,\cdot)$ the scalar product in $L^2(\xR)$. To estimate the $\dot{H}^s$-norm of $f_n$, we 
make act $\Lambda^s$ on the equation, and then take its scalar product with $\Lambda^s f_n$. We get
$$
(\partial_t \Lambda^s f_n,\Lambda^{s}f_n)+(\Lambda^{s+1} f_n,\Lambda^{s} f_n)
= \big(\Lambda^{s}\mathcal{T}_n(f_n),\Lambda^{s}f_n\big).
$$
Since the Muskat equation is parabolic of order one, we will be able to gain one half-derivative. 
We exploit this parabolic regularity 
by writing that
$$
(\Lambda^{s+1} f_n,\Lambda^{s} f_n)=
\blA  f_n\brA_{\dot{H}^{s+\mez}}^2,
$$
and
\begin{align*}
\big( \Lambda^{s}\mathcal{T}_n(f_n),\Lambda^{s}f_n\big)
&=\big(\Lambda^{s}J_n\big(\mathcal{T}(f_n)f_n\big),\Lambda^{s}f_n\big)\\
&=\big(\Lambda^{s}\big(\mathcal{T}(f_n)f_n\big),J_n\Lambda^{s}f_n\big)\\
&=\big(\Lambda^{s}\big(\mathcal{T}(f_n)f_n\big),\Lambda^{s}f_n\big) \quad\text{since }J_nf_n=f_n,\\
&=\big(\Lambda^{s-\mez} \mathcal{T}(f_n)f_n, \Lambda^{s+\mez}f_n\big).
\end{align*}
Consequently, we find 
$$
\mez  \fract\Vert f_n \Vert^{2}_{\dot H^{s}} +\blA  f_n\brA_{\dot{H}^{s+\mez}}^2 =  
\big(\Lambda^{s-\mez}\mathcal{T}(f_n)f_n, \Lambda^{s+\mez}f_n\big).
$$
We next use a variant of the paralinearization formula given by Corollary~\ref{Coro:para}. 
Set 
$$
\eps=s-\tdm.
$$
We claim that, for any function $g$,
$$
\Lambda^{1+\eps}(T(g)g)=V(g)\partial_x \Lambda^{1+\eps} g+\frac{g_x^2}{1+g_x^2}\Lambda^{2+\eps}g+\Lambda^{1+\eps}R_\eps(g)
$$
where $V(g)$ and $R_\eps(g)$ are two functions satisfying, for any fixed $\nu<\eps$,
\begin{align}
\lA R_\eps(g)\rA_{\dot{H}^{1+\eps}}&\le 
\mathcal{F}(\lA g\rA_{\dot{H}^1\cap \dot{H}^{s}})\lA g\rA_{\dot{H}^{s+\mez-\frac{\eps}2}},\label{n1001}\\
\lA V(g)\rA_{C^{0,\nu}}&\le \mathcal{F}(\lA g\rA_{\dot{H}^1\cap \dot{H}^{s}})\lA g\rA_{\dot{H}^{s+\mez-\frac{\eps}2}}.
\label{n1002}
\end{align}
where $\mathcal{F}$ depends only on $\eps$ (that is $s$) and $\nu$ (which will be specified later). 
The proof of this claim is similar to the one of \e{coro:paraeq}. 

With notations as above, set
$$
V_n=V(f_n),\quad R_n=\Lambda^{1+\eps}R_{\eps}(f_n), \quad \gamma_n=\frac{f_{nx}^2}{1+f_{nx}^2}
\quad\text{where }f_{nx}=\partial_xf_n.
$$
Then,
\be\label{n100}
\begin{aligned}
\frac{1}{2}  \fract\Vert f_n \Vert^{2}_{\dot H^{s}}+\lA f_n\rA_{\dot{H}^{s+\mez}}^2
&=  \big( \gamma_n \Lambda^{s+\mez}f_n,\Lambda^{s+\mez}f_n , \big)\\
&\quad +\Big(\big(V_n\partial_x \Lambda^{s-\mez}f_n+R_n\big), \Lambda^{s+\mez}f_n \Big).
\end{aligned}
\ee
Now the key point is that
$$
\lA f_n\rA_{\dot{H}^{s+\mez}}^2
-  \big(\gamma_n\Lambda^{s+\mez}f_n, \Lambda^{s+\mez}f_n \big)
= \int  \frac{\big( \Lambda^{s+\mez}f_n \big)^2}{1+f_{nx}^2} \dx.
$$
On the other hand, 
the Cauchy-Schwarz inequality and the estimate~\e{n1001} imply that
\begin{align*}
\big\vert \big( R_n,\Lambda^{s+\mez}f_n\big)\big\vert &\le \lA R_n\rA_{L^2}\blA \Lambda^{s+\mez}f_n\brA_{L^2}\\
&\le 
\mathcal{F}(\lA f_n\rA_{\dot{H}^1\cap \dot{H}^{s}})\lA f_n\rA_{\dot{H}^{s+\mez-\frac{\eps}2}}\lA f_n\rA_{\dot{H}^{s+\mez}}.
\end{align*}

It remains to estimate the contribution of $V_n$ to the 
second term in the right-hand side of \e{n100}. 
Here we use the commutator estimate given by Lemma~\ref{L:Hilbert}. 
To do so, one uses the identity $\mathcal{H}\Lambda=-\partial_x$  where $\mathcal{H}$ is the Hilbert transform, to write
$$
V_n\partial_x \Lambda^{s-\mez}f_n=-V_n \Lambda^{s+\mez}\mathcal{H}f_n.
$$ 
Since $\mathcal{H}$ is skew-symmetric, we deduce that
$$
\Big(  V_n\partial_x \Lambda^{s-\mez}f_n,\Lambda^{s+\mez}f_n\Big)
=\mez \Big(\big[ \mathcal{H},V_n\big]\Lambda^{s+\mez}f_n,\Lambda^{s+\mez}f_n \Big).
$$
Now we exploit the regularity result for $V_n$ given by~\e{n1002}. Fix $\nu=2\eps/3$ and $\theta=\eps/2$. 
By applying the commutator estimate in Lemma~\ref{L:Hilbert}, we obtain
\begin{align*}
\blA \big[\mathcal{H},V_n\big]\Lambda^{s+\mez}f_n\brA_{L^2}
&\les \lA V_n\rA_{C^{0,\nu}}\blA \Lambda^{s+\mez}f_n\brA_{H^{-\theta}}\\
&\les \lA V_n\rA_{C^{0,\nu}}\lA f_n\rA_{\dot{H}^{s+\mez-\frac{\eps}2}}\\
&\le \mathcal{F}(\lA f_n\rA_{\dot{H}^1\cap \dot{H}^{s}})\lA f_n\rA_{\dot{H}^{s+\mez-\frac{\eps}2}}.
\end{align*}

So, by combining the above estimates, 
$$
\frac{1}{2}  \fract\Vert f_n \Vert^{2}_{\dot H^{s}} +\int  \frac{\big( \Lambda^{s+\mez}f_n \big)^2}{1+f_{nx}^2} \dx
\le \mathcal{F}(\lA f_n\rA_{\dot{H}^1\cap \dot{H}^{s}})\lA f_n\rA_{\dot{H}^{s+\mez-\frac{\eps}2}}\lA f_n\rA_{\dot{H}^{s+\mez}}.
$$
The end of the proof will consist in exploiting the parabolic regularity and a variant of Gronwall's lemma 
to absorb the right-hand side. Set $K(t) =\sup_{x\in \xR}\la \partial_x f_n(t,x)\ra$. 
Then
$$
\frac{1}{1+f_{nx}^2}\ge \frac{1}{1+K^2},
$$
so 
\be\label{n101}
\begin{aligned}
\frac{1}{2}  \fract\Vert f_n \Vert^{2}_{\dot H^{s}} +\frac{C}{1+K^2}\lA f_n\rA_{\dot{H}^{s+\mez}}^2
\le \mathcal{F}(\lA f_n\rA_{\dot{H}^1\cap \dot{H}^{s}})\lA f_n\rA_{\dot{H}^{s+\mez-\frac{\eps}2}}\lA f_n\rA_{\dot{H}^{s+\mez}}.
\end{aligned}
\ee
Then we observe that
\begin{multline*}
\mathcal{F}(\lA f_n\rA_{\dot{H}^1\cap \dot{H}^{s}})\lA f_n\rA_{\dot{H}^{s+\mez-\frac{\eps}2}}\lA f_n\rA_{\dot{H}^{s+\mez}}\\
\le \frac{C}{2(1+K^2)}\lA f_n\rA_{\dot{H}^{s+\mez}}^2+\frac{1+K^2}{2C}
\mathcal{F}(\lA f_n\rA_{\dot{H}^1\cap \dot{H}^{s}})^2\lA f_n\rA_{\dot{H}^{s+\mez-\frac{\eps}2}}^2.
\end{multline*}
Since $K\les \lA f_n\rA_{\dot{H}^1\cap \dot{H}^{s}}$ by Sobolev embedding, 
up to modifying the value of the function $\mathcal{F}$, 
by inserting the above inequality in~\e{n101}, we get
\be\label{n101-bis}
\begin{aligned}
\frac{1}{2}  \fract\Vert f_n \Vert^{2}_{\dot H^{s}} +\frac{C}{2(1+K^2)}\lA f_n\rA_{\dot{H}^{s+\mez}}^2
\le \mathcal{F}(\lA f_n\rA_{\dot{H}^1\cap \dot{H}^{s}})\lA f_n\rA_{\dot{H}^{s+\mez-\frac{\eps}2}}^2.
\end{aligned}
\ee
To conclude, it will suffice to replace in the right side the norm 
$\lA f_n\rA_{\dot{H}^{s+\mez-\frac{\eps}2}}^2$ by $\lA f_n\rA_{\dot{H}^{s}}^2$. 
To do this, we begin by using the interpolation inequality:
\be\label{n1503}
\lA f_n\rA_{\dot{H}^{s+\mez-\frac{\eps}2}}^2\le \lA f_n\rA_{\dot{H}^{s}}^{2\theta}
\lA f_n\rA_{\dot{H}^{s+\mez}}^{2-2\theta},
\ee
for some $\theta\in (0,1)$. Next, because of the Young's inequality
\be\label{n1504}
xy\le \frac{1}{p}x^p+\frac{1}{p'}y^{p'}\quad\text{with}\quad \frac{1}{p}+\frac{1}{p'}=1,
\ee
applied with $p=2/(2-2\theta)$, we infer that
\be\label{n1505}
\frac{1}{2}  \fract\Vert f_n \Vert^{2}_{\dot H^{s}} +\frac{C}{4(1+K^2)}\lA f_n\rA_{\dot{H}^{s+\mez}}^2
\le \mathcal{F}(\lA f_n\rA_{\dot{H}^1\cap \dot{H}^{s}})\lA f_n\rA_{\dot{H}^{s}}^2,
\ee
where as above we modified the value of the function $\mathcal{F}$. 
From this, it is now an easy matter to obtain the conclusion of the proposition. 
Firstly, integration of the above estimate gives
\begin{multline*}
\mez\Vert f_n(t) \Vert^{2}_{\dot H^{s}} +\frac{C}{4(1+K^2)}\int_0^t \lA f_n(t')\rA_{\dot{H}^{s+\mez}}^2\dt'\\
\le \mez\Vert f_n(0) \Vert^{2}_{\dot H^{s}}+t\sup_{t'\in [0,t]} \mathcal{F}(\lA f_n(t')\rA_{\dot{H}^1\cap \dot{H}^{s}})\lA f_n(t')\rA_{\dot{H}^{s}}^2.
\end{multline*}
Modifying $\mathcal{F}(\cdot)$ and $C$, we deduce that
$$
\Vert f_n(t) \Vert^{2}_{\dot H^{s}} +\frac{C}{1+K^2}\int_0^t \lA f_n(t')\rA_{\dot{H}^{s+\mez}}^2\dt'
\le \Vert f_n(0) \Vert^{2}_{\dot H^{s}}
+T\mathcal{F}\Big(\sup_{t\in[0,T]}\lA f_n(t)\rA_{\dot{H}^1\cap \dot{H}^s}^2\Big),
$$
for any $t\in [0,T]$. By taking the supremum over $t\in [0,T]$, we deduce an estimate for 
$\sup_{t\in [0,T]}\lA f_n(t)\rA_{\dot{H}^s}^2$. Now, the desired estimate for 
$\sup_{t\in [0,T]}\lA f_n(t)-f_0\rA_{\dot{H}^s}^2$ follows from the triangle inequality 
and the fact that $\Vert f_n(0) \Vert_{\dot H^{s}}\le \Vert f_0 \Vert_{\dot H^{s}}$.
\end{proof}

We will also need another energy estimate to compare 
two different solutions $f_1$ and $f_2$. The main difficulty here 
will be to find the optimal space in which one can perform an energy estimate. 
The most simpler way to do so would be to estimate their 
difference $f_1-f_2$ in the biggest possible space and to 
use an interpolation inequality to control the latter 
in a space of smoother function. 
This suggests to estimate $f_1-f_2$ in $C^0([0,T];L^2(\xR))$. On the other hand, by thinking 
of the fluid problem, we might think that it is compulsory to control the difference between the two functions parametrizing the two free surfaces 
in a space of smooth functions. We will see later that, somewhat unexpectedly, that it is enough to 
estimate $f_1-f_2$ in $C^0([0,T];\dot{H}^{1/2}(\xR))$.  
In this direction, we will use the following proposition.

\begin{prop}\label{P:energyL2}
$i)$ For all $s$ in $(3/2,2)$, there exists a non-decreasing function $\mathcal{F}\colon \xR_+\rightarrow \xR_+$ 
such that, for any $n\in \xN\setminus\{0\}$, any $T>0$, and any functions
\begin{align*}
f&\in C^0([0,T];\dot{H}^1(\xR)\cap\dot{H}^s(\xR)), \\
g&\in C^1([0,T];H^\mez(\xR))\quad\text{with }J_ng=g,\\
F&\in C^0([0,T];L^2(\xR)),
\end{align*}
satisfying the equation
\begin{equation}\label{n1508}
\partial_tg-J_n\big(V(f)\partial_x g)+J_n\Big(\frac{1}{1+f_x^2}\Lambda g\Big)=F,
\ee
where $V(f)$ is as above, we have the estimate
\begin{equation}\label{n1507}
\frac{1}{2}  \fract\Vert g \Vert^{2}_{\dot{H}^\mez} +\int  \frac{(\Lambda g)^2}{1+f_x^2} \dx
\le \mathcal{F}(\lA f\rA_{\dot{H}^1\cap \dot{H}^{s}})\lA  g\rA_{\dot{H}^{1-\frac{\eps}2}}\lA g\rA_{\dot{H}^1}
+( F,\Lambda g)_{L^2}.
\end{equation}
where $\eps=s-3/2$ and $C=\mathcal{F}\big(\lA f\rA_{L^\infty([0,T];\dot{H}^1\cap\dot{H}^s)}\big)$.

$ii)$ Moreover, the same result is true when one replaces $J_n$ by the identity.
\end{prop}
\begin{proof}
To prove~\e{n1507} we take the $L^2$-scalar product of the equation \e{n1508} with $\Lambda g$. 
Since
$$
\frac{1}{2}  \fract\Vert g \Vert^{2}_{\dot{H}^\mez}=(\partial_t g,\Lambda g),\qquad 
\Big(J_n\Big(\frac{1}{1+f_x^2}\Lambda g\Big),\Lambda g\Big)=\int  \frac{( \Lambda g)^2}{1+f_x^2} \dx,
$$
(where we used $J_ng=g$), 
we only have to estimate $\big(J_n\big(V(f)\partial_x g),\Lambda g\big)$. As above, 
writing $\partial_x=-\mathcal{H}\Lambda$, where $\mathcal{H}$ is the Hilbert transform satisfying 
$\mathcal{H}^*=-\mathcal{H}$, we obtain
$$
\la \big(J_n\big(V(f)\partial_x  g),\Lambda g\big)\ra
=\mez\la \big( \big[ \mathcal{H},V(f)\big]\Lambda g,\Lambda g\big)\ra.
$$
Set $\eps=s-3/2$, $\nu=2\eps/3$ and $\theta=\eps/2$. We use Lemma~\ref{L:Hilbert} to obtain
\begin{align*}
\la \big(J_n\big(V(f)\partial_x  g),\Lambda g\big)\ra
&\le \mez\lA \big[\mathcal{H},V(f)\big]\Lambda g\rA_{L^2} \lA \Lambda g\rA_{L^2}\\
&\le \lA V(f)\rA_{C^{0,\nu}}\lA \Lambda g\rA_{H^{-\theta}} \lA \Lambda g\rA_{L^2}\\
&\le \mathcal{F}(\lA f\rA_{\dot{H}^1\cap \dot{H}^{s}}) \lA \Lambda g\rA_{H^{-\frac{\eps}{2}}}
\lA \Lambda g\rA_{L^2}\\
&\le \mathcal{F}(\lA f\rA_{\dot{H}^1\cap \dot{H}^{s}})\lA g\rA_{\dot{H}^{1-\frac{\eps}2}}\lA g\rA_{\dot{H}^1}.
\end{align*}
This completes the proof of $i)$ and the same arguments can be used to prove $ii)$.
\end{proof}

\subsection{End of the proof}

In this paragraph we complete the analysis of the Cauchy problem. 
We begin by proving the uniqueness part in Theorem~\ref{T1}.
\begin{lemma}\label{L:uniqueness}
Assume that $f$ and $f'$ are two solutions of the Muskat equation with the same initial data and satisfying the assumptions of Theorem~\ref{T:Cauchy}. 
Then $f=f'$. 
\end{lemma}
\begin{proof}
Set
$$
g=f-f',\quad M=\lA f\rA_{L^\infty([0,T];\dot{H}^1\cap \dot{H}^s)}
+\lA f'\rA_{L^\infty([0,T];\dot{H}^1\cap \dot{H}^s)}.
$$
We denote by $C(M)$ various constants depending only on $M$. 

We want to prove that $g=0$. To do so, we use the energy estimate in $\dot{H}^{1/2}(\xR)$. 
The key point is to write that $g$ is a smooth function, in $C^1([0,T];H^\mez(\xR))$, satisfying
$$
\partial_t g+\Lambda g =\mathcal{T}(f)g+F_1
\quad\text{with}\quad F_1=\mathcal{T}(f)f'-\mathcal{T}(f')f'.
$$
This term is estimated by means of point~$\ref{Prop:low3})$ in Proposition~\ref{P:continuity} 
with $\delta=\eps=s-3/2$, 
\be\label{n1501}
\lA F_1\rA_{L^2}=\lA (\mathcal{T}(f)-\mathcal{T}(f'))f'\rA_{L^2}
\le C \lA f-f'\rA_{\dot{H}^{1-\eps}}\lA f'\rA_{\dot{H}^{\tdm+\eps}}=C(M)\lA g\rA_{\dot{H}^{1-\eps}}.
\ee
Recall from \e{n3} that 
$$
\mathcal{T}(f)g=\frac{f_x^2}{1+f_x^2}\Lambda g+V(f)\partial_x g+R(f,g)
$$
where $R(f,g)$ satisfies (setting $\eps=s-3/2$),
\be\label{n1502}
\lA R(f,g)\rA_{L^2}\le C\lA f\rA_{\dot{H}^{\tdm+\eps}}\lA g\rA_{\dot{B}^{1-\eps}_{2,1}}\le C(M)\lA g\rA_{\dot{B}^{1-\eps}_{2,1}}.
\ee
Therefore, $g$ satisfies
$$
\partial_t g -V\partial_x g+\frac{1}{1+f_x^2}\Lambda g=F
$$
where $F=F_1+R(f)g$. In view of the estimates~\e{n1501},~\e{n1502} and 
the embedding 
$\dot{H}^{1-3\eps/2}(\xR)\cap \dot{H}^{1-\eps/2}(\xR)\hookrightarrow \dot{B}^{1-\eps}_{2,1}$ (see Lemma~\ref{L:B21}), 
we have
$$
\la (F,\Lambda g)\ra\le \lA F\rA_{L^2}\lA g\rA_{\dot{H}^1}\le C(M) \lA g\rA_{\dot{H}^{1-\frac{\eps}2}\cap\dot{H}^{1-\frac{3\eps}2}}\lA g\rA_{\dot{H}^1}.
$$
Hence, it follows from Proposition~\ref{P:energyL2} (see point $ii)$) that
$$
\frac{1}{2}  \fract\Vert g \Vert^{2}_{\dot{H}^\mez} +\int  \frac{( \Lambda  g )^2}{1+f_x^2} \dx
\le C(M) \lA g\rA_{\dot{H}^{1-\frac{\eps}2}\cap\dot{H}^{1-\frac{3\eps}2}}\lA g\rA_{\dot{H}^1}.
$$
Next, 
we use interpolation inequalities as in the proof of Proposition~\ref{Prop:aprioriHs}. More precisely, 
by using arguments parallel to those used to deduce \e{n1505} from \e{n1503}-\e{n1504}, we get
$$
\frac{1}{2}  \fract\Vert g \Vert^{2}_{\dot{H}^\mez} +\frac{1}{4(1+\lA f_x\rA_{L^\infty_{t,x}})}\int ( \Lambda g )^2 \dx
\le C(M)\lA g\rA_{\dot{H}^\mez}^2.
$$
This obviously implies that
$$
\frac{1}{2}  \fract\Vert g \Vert^{2}_{\dot{H}^\mez} \le C(M)\lA g\rA_{\dot{H}^\mez}^2.
$$
Since $g(0)=0$, the Gronwall's inequality implies that $\lA g\rA_{\dot{H}^\mez}=0$ so $g=0$, which completes the proof.
\end{proof}

Having proved the uniqueness of solutions, we now study their existence. The key step will be here 
to apply the {\em a priori} estimates proved in Proposition~\ref{Prop:aprioriHs}. This will 
give us uniform bounds for the solutions $f_n$ defined in $\S\ref{S:approximate}$.

\begin{lemma}
There exists $T_0>0$ such that $T_n\ge T_0$ for all $n\in \xN\setminus\{0\}$ 
and such that $(f_n-f_0)_{n\in\xN}$ is bounded in 
$C^{0}([0,T_0]; H^{s}(\xR))$. 
\end{lemma}
\begin{proof}
We use the notations of $\S\ref{S:approximate}$ and Proposition~\ref{Prop:aprioriHs}. 
Given $T<T_n$, we define
$$
M_n(T)=\sup_{t\in [0,T]}\lA f_n(t)-f_0\rA_{L^2\cap \dot{H}^s}^2,\qquad 
N_n(T)=M_n(T)+\lA f_0\rA_{\dot{H}^1\cap \dot{H}^s}^2.
$$
Denote by $\mathcal{F}$ the function whose existence is the conclusion of Proposition~\ref{Prop:aprioriHs} and 
set 
$$
A=10\lA f_0\rA_{\dot{H}^1\cap \dot{H}^s}^2.
$$
We next pick $0<T_{0}\le 1$ small enough such that
$$
3(1+T_0)^2\lA f_0\rA_{\dot{H}^1\cap \dot{H}^s}^2+T_0\mathcal{F}(A)< A,
$$
We claim the uniform bound
\begin{equation*}
\forall n\in\xN\setminus\{0\},~\forall T\in I_n\defn [0,\min\{T_{0},T_n\}),\quad N_n(T) <A.
\end{equation*}
Let us prove this claim by contradiction. Assume that for some $n$ 
there exists $\tau_n\in I_n$ such that $N_n(\tau_n)\ge A$ and consider 
the smallest of such times (then $\tau_n>0$ since $T\mapsto N_n(T)$ is continuous and 
$N_n(0)<A$ by construction). 
Then, by definition, for all $0<T\le \tau_n$, one has
$N_n(T)\le A$ and $N_n(\tau_n)=A$. 
Since $\lA \partial_xf_n(t)\rA_{L^\infty(\xR)}\les \lA f_n(t)\rA_{\dot{H}^1\cap\dot{H}^s}$ (see~\e{o1}),  
we have a uniform control of the $L^\infty_{x}$-norm 
of $\partial_xf_n$ on $[0,T]$ in terms of $A$, hence 
we are in position to apply the {\em a priori} estimate~\e{apriori}. 
Now, if we add $\lA f_0\rA_{\dot{H}^1\cap \dot{H}^s}^2$ to both sides of \e{apriori} 
we deduce that
$$
N_n(\tau_n)+\frac{C}{1+C(A)^2}\int_{0}^{\tau_n}\lA f_n(t)\rA_{\dot{H}^{s+\mez}}^2\dt
\le 3(1+\tau_n)^2\lA f_0\rA_{\dot{H}^1\cap \dot{H}^s}^2+\tau_n\mathcal{F}(N_n(\tau_n)).
$$
We infer that
\begin{align*}
A=N_n(\tau_n)&\le 3(1+\tau_n)^2\lA f_0\rA_{\dot{H}^1\cap \dot{H}^s}^2+\tau_n\mathcal{F}(N_n(\tau_n))\\
&\le 3(1+T_0)^2\lA f_0\rA_{\dot{H}^1\cap \dot{H}^s}^2+T_0\mathcal{F}(A)\\
&<A,
\end{align*}
hence the contradiction. We thus have proved that, for all $n\in \xN$ and all $T\le \min\{T_{0},T_n\}$, we have
$$
\sup_{t\in [0,T]}\lA f_n(t)-f_0\rA_{L^2\cap \dot{H}^s}^2\le A.
$$
This obviously implies that 
$$
\sup_{t\in [0,T]}\lA f_n(t)-f_0\rA_{L^2}\le \sqrt{A}.
$$
Since 
$$
u_n=f_n-J_nf_0=f_n-f_0+(I-J_n)f_0,
$$
and since $\lA (I-J_n)f_0\rA_{L^2}\le \lA f_0\rA_{\dot{H}^1}$, the previous bound 
implies that the norm $\lA u_n(t)\rA_{L^2}$ is bounded for all $t\le \min\{T_{0},T_n\}$. 
The alternative~\e{A4} then implies that the lifespan of $f_n$ is bounded from below by $T_0$. 
And the previous inequality shows that $(f_n-f_0)$ is bounded in $C^0([0,T_0];H^s(\xR))$. 
This completes the proof.
\end{proof}

At that point, we have defined a sequence $(f_n)$ of solutions to well-chosen 
approximate systems. The next task is to prove that this sequence converges. 
Here a word of caution is in order: $\dot{H}^s(\xR)$ 
is not a Banach space when $s>1/2$. 
To overcome this difficulty, we use the fact that $u_n=f_n-f_0$ is bounded in $C^0([0,T_0];H^s(\xR))$, 
where $H^s(\xR)$ is the nonhomogeneous space $L^2(\xR)\cap \dot{H}^s(\xR)$, 
which is a Banach space. We claim that, in addition, $(u_n)$ is a Cauchy sequence 
in $C^0([0,T_0];H^{s'}(\xR))$ for any $s'<s$. Let us assume this claim for the moment. 
This will imply that $(u_n)$ converges in the latter 
to some limit $u$. Now, setting $f=f_0+u$ and using the continuity result for $\mathcal{T}(f)f$ given by $\ref{Prop:low2})$ in Proposition~\ref{P:continuity}, we verify immediately that $f$ is a solution to the Cauchy problem for the 
Muskat equation. It would remain to prove that $u$ is continuous 
in time with values in $H^s(\xR)$ (instead of $H^{s'}(\xR)$ for any $s'<s$). For the sake of shortness, 
this is the only point that we do not prove in details in this paper (referring to~\cite{ABZ3} 
for the proof of a similar result in a case with similar difficulties). 

To conclude the proof of Theorem~\ref{T:Cauchy}, it remains only to establish the following

\begin{lemma}
For any real number $s'$ in $[0,s)$, the sequence $(u_n)$ is a Cauchy sequence in $C^0([0,T_0];H^{s'}(\xR))$.
\end{lemma}
\begin{proof}
The proof is in two steps. 
We begin by proving that $(f_n)$ is a Cauchy sequence in $C^0([0,T_0];\dot{H}^{s'}(\xR))$ 
for $1/2\le s'<s$. Then, we use this result and an elementary $L^2$-estimate to infer that $(u_n)$ is a Cauchy sequence in $C^0([0,T_0];L^2(\xR))$.

By using estimates parallel to those used to prove Lemma~\ref{L:uniqueness}, one obtains that 
$(f_n)$ is a Cauchy sequence in $C^0([0,T_0];\dot{H}^\mez(\xR))$. Now consider $1/2<s'<s$. By interpolation, 
there exists $\alpha$ in $(0,1)$ such that
$$
\lA u\rA_{\dot{H}^{s'}}\les \lA u\rA_{\dot{H}^\mez}^\alpha \lA u\rA_{\dot{H}^s}^{1-\alpha}.
$$
Consequently, since $(f_n)$ is bounded 
in $C^0([0,T_0];\dot{H}^s(\xR))$, we deduce that $(f_n)$ is a Cauchy sequence in $C^0([0,T_0];\dot{H}^{s'}(\xR))$ for any $s'<s$. 

It remains only to prove that $(u_n)$ is a Cauchy 
sequence in $C^0([0,T_0];L^2(\xR))$. To do so, we proceed differently. 
Starting from (see \e{n1509}), 
\begin{equation*}
\partial_t f_n+\Lambda f_n=J_n\mathcal{T}(f_n)f_n,
\end{equation*}
we obtain that $u_n-u_p=f_n-f_p-(J_n-J_p)f_0$ satisfies
\be\label{n1701}
\partial_t (u_n-u_p)+\Lambda(u_n-u_p)=F_{np}+G_{np}
\ee
where
$$
F_{np}=J_n(\mathcal{T}(f_n)f_n-\mathcal{T}(f_p)f_p),\quad G_{np}=(J_n-J_p)\big(-\Lambda f_0+\mathcal{T}(f_p)f_p\big).
$$
We now use an elementary $L^2$-estimate. We take the $L^2$-scalar product of the equation~\e{n1701} with $u_n-u_p$, to obtain, 
since $u_n(0)-u_p(0)=0$,
$$
\sup_{t\in [0,T]}\lA u_n(t)-u_p(t)\rA_{L^2}\le \lA F_{np}\rA_{L^1([0,T];L^2)}
+\lA G_{np}\rA_{L^2([0,T];\dot{H}^{-\mez})}.
$$
So it remains only to prove that $\lA F_{np}\rA_{L^1([0,T];L^2)}$ and 
$\lA G_{np}\rA_{L^2([0,T];\dot{H}^{-\mez})}$ are arbitrarily small for $n,p$ large enough. 
Here we use the result proved in the first part of the proof. Namely, since $(f_n)$ 
is a Cauchy sequence in $C^0([0,T];\dot{H}^1(\xR)\cap \dot{H}^\tdm(\xR))$, 
we deduce from point $\ref{Prop:low2})$ in Proposition~\ref{P:continuity} that 
$$
\mathcal{T}(f_n)f_n-\mathcal{T}(f_p)f_p
$$
is small in $C^0([0,T];L^2(\xR))$ for $n,p$ large enough. 
On the other hand, using the estimate
$$
\lA (J_n-J_p)u\rA_{\dot{H}^{-\mez}}\le\frac{1}{\sqrt{\min(n,p)}}\lA u\rA_{L^2},
$$
we verify that 
$\lA G_{np}\rA_{L^2([0,T];\dot{H}^{-\mez})}$ is arbitrarily small for $n,p$ large enough. 
This completes the proof.
\end{proof}

\section*{Acknowledgments} 
\noindent  T.\ A. acknowledges the support of the SingFlows project, grant ANR-18-CE40-0027 
of the French National Research Agency (ANR). O.\ L.  has been partially supported by the National Grant MTM2014-59488-P from 
the Spanish government and the ERC through the Starting Grant project H2020-EU.1.1.-63922

\vspace{3mm}

\noindent\textbf{Thomas Alazard}\\
\noindent CNRS and CMLA, \'Ecole Normale Sup{\'e}rieure de Paris-Saclay, France

\vspace{3mm}
\noindent\textbf{Omar Lazar}\\
\noindent Departamento de An\'alisis Matem\'atico \& IMUS, Universidad de Sevilla, Spain

\end{document}